\newcommand{\MailFile}[1]{}
\documentclass[a4paper,12pt]{article}
\usepackage{a4wide}
\usepackage[utf8]{inputenc}
\usepackage{amsmath}
\usepackage{graphicx}
\usepackage{color}
\usepackage{url}

\usepackage{email}
\MailFile{two_places_arxiv.tex}
\DoNotMail{article.cls,size12.clo,a4wide.sty,a4.sty,import.sty,inputenc.sty,utf8.def,*dfu,ams*,graphic*,keval.sty,trig.sty,pdftex.def,color*,url.sty,ifthen.sty,supp-pdf.tex,*fd,*bbl}
\DoMail{ubmexs.fd}

\usepackage[LaTeXeqno]{diagrams}
\newarrow{ToFro}<--->

\usepackage{amsthm}
\newtheorem{thm}{Theorem}[section]
\newtheorem{lem}[thm]{Lemma}
\newtheorem{cor}[thm]{Corollary}
\newtheorem{prop}[thm]{Proposition}
\newtheorem{obs}[thm]{Observation}

\newtheorem*{RankConjecture}{Rank Conjecture}
\newtheorem*{MainTheorem}{Main Theorem}
\newtheorem*{TheoremGeom}{Theorem \ref{thm:main_theorem_geometric_formulation}}
\newtheorem*{TheoremAlg}{Theorem \ref{thm:main_theorem_algebraic_formulation}}

\theoremstyle{definition}
\newtheorem{exmpl}[thm]{Example}

\theoremstyle{remark}
\newtheorem*{rem}{Remark}
\newtheorem{remnum}[thm]{Remark}

\newcommand{\real}{\mathbb{R}}

\newcommand{\integer}{\mathbb{Z}}
\newcommand{\field}{\mathbb{F}}
\newcommand{\nat}{\mathbb{N}}

\numberwithin{equation}{section}

\newcommand{\calA}{\mathcal{A}}

\newcommand{\calG}{\mathcal{G}}

\newcommand{\calO}{\mathcal{O}}

\newcommand{\abs}[1]{\left\lvert#1\right\rvert}

\newcommand{\conv}{\operatorname{conv}}
\newcommand{\gen}[1]{\left\langle{}#1\right\rangle}
\newcommand{\lk}{\operatorname{lk}}
\newcommand{\pr}{\operatorname{pr}}
\newcommand{\typ}{\operatorname{typ}}
\newcommand{\st}{\operatorname{st}}

\DeclareMathOperator{\SL}{SL}

\newcommand{\Aut}{\operatorname{Aut}}
\newcommand{\relint}{\operatorname{relint}}

\usepackage{amssymb}
\usepackage{stmaryrd}
\newcommand{\horizontal}{\multimap}

\newcommand{\direction}{\rightslice}

\usepackage{buxsym}

\newcommand{\nin}{\notin}
\newcommand{\union}{\cup}
\newcommand{\Union}{\bigcup}
\newcommand{\intersect}{\cap}
\newcommand{\Intersect}{\bigcap}
\renewcommand{\Join}{\bigast}

\newcommand{\scaprod}[2]{\left\langle{}#1\mid{}#2\right\rangle}

\newcommand{\defeq}{\mathrel{\mathop{:}}=}

\newcommand{\op}{\mathrel{\operatorname{op}}}
\newcommand{\opm}{\operatorname{op}}

\newcommand{\optionalsubsuper}[3]{%
\ifthenelse{\equal{#2}{}}{%
  \ifthenelse{\equal{#3}{}}{%
    #1}{%
    #1^{#3}}
  }{%
  \ifthenelse{\equal{#3}{}}{%
    #1_{#2}}{%
    #1_{#2}^{#3}}}}
\usepackage{xargs}
\usepackage{ifthen}

\newcommand{\SumOfLocalRanks}{m}
\newcommand{\Dimension}{n}
\newcommand{\GlobalField}{k}
\newcommand{\Places}{S}
\newcommand{\Place}{s}
\newcommand{\Valuation}{\nu}
\newcommand{\Class}[1]{[#1]}
\newcommandx{\Integers}[1][1=\Places]{\calO_{#1}}
\newcommandx{\FiniteField}[1][1={}]{\field_{#1}}
\newcommandx{\GroupScheme}[1][1={}]{\calG\ifthenelse{\equal{#1}{}}{}{(#1)}}
\newcommand{\Group}{G}
\newcommand{\TopFin}[1]{F_{#1}}
\newcommandx{\EMSpace}[2][2=1]{K(#1,#2)}
\newcommand{\Another}[1]{{#1}'}
\newcommand{\YetAnother}[1]{{#1}''}
\newcommand{\Infty}[1]{#1^{\infty}}
\newcommandx{\Ray}[2][1={},2={}]{\optionalsubsuper{\rho}{#1}{#2}}

\newcommand{\Geodesic}{\gamma}
\newcommand{\Parameter}{r}
\newcommand{\Constant}{c}
\newcommandx{\Proj}[2][1={},2={}]{\optionalsubsuper{\pr}{#1}{#2}}
\newcommand{\Busemann}{\beta}
\newcommandx{\ReparBusemann}[1][1={}]{{\xi_{#1}}}
\newcommand{\EuclBuilding}{X}
\newcommand{\PosEuclBuilding}{\EuclBuilding_+}
\newcommand{\NegEuclBuilding}{\EuclBuilding_-}

\newcommand{\EuclTwinBuilding}{(\PosEuclBuilding,\NegEuclBuilding)}
\newcommand{\EuclDistance}{\mu}

\newcommandx{\EuclCoDistance}[1][1={}]{\optionalsubsuper{\EuclDistance^*}{#1}{}}
\newcommand{\SpherDistance}{d}

\newcommand{\Apartment}{\Sigma}
\newcommand{\PosApartment}{\Apartment_+}
\newcommand{\NegApartment}{\Apartment_-}
\newcommand{\TwinApartment}{(\PosApartment,\NegApartment)}
\newcommand{\EuclSpace}{\mathbb{E}}
\newcommand{\Space}{Y}
\newcommand{\Vector}{v}
\newcommand{\AltVector}{w}
\newcommand{\FaceVector}{f}
\newcommand{\NormalVector}{n}
\newcommand{\PointAtInfty}{\xi}

\newcommand{\Point}{x}
\newcommand{\AltPoint}{y}
\newcommand{\PosPoint}{x_+}
\newcommand{\NegPoint}{x_-}
\newcommand{\SpherPoint}{p}
\newcommand{\Weyl}{W}
\newcommand{\DummyType}{X}
\newcommandx{\Zonotope}[1][1={}]{Z\ifthenelse{\equal{#1}{}}{}{(#1)}}
\newcommand{\Polytope}{P}
\newcommand{\Face}{F}
\newcommand{\NormalCone}[1]{N(#1)}
\newcommand{\Directions}{D}
\newcommand{\Direction}{z}
\newcommand{\ConvexSet}{M}
\newcommand{\ConvexSetPoint}{m}
\newcommandx{\InftyRay}[3][3={}]{\optionalsubsuper{[#1,#2)}{#3}{}}
\newcommand{\Cell}{\sigma}
\newcommand{\PosCell}{\Cell_+}
\newcommand{\NegCell}{\Cell_-}
\newcommand{\BigCell}{\tau}
\newcommand{\PosBigCell}{\BigCell_+}
\newcommand{\NegBigCell}{\BigCell_-}

\newcommand{\Vertex}{v}
\newcommand{\PosVertex}{\Vertex_+}
\newcommand{\NegVertex}{\Vertex_-}
\newcommand{\AltVertex}{w}

\newcommand{\Chamber}{C}
\newcommand{\PosChamber}{\Chamber_+}
\newcommand{\NegChamber}{\Chamber_-}
\newcommand{\AltChamber}{D}
\newcommand{\Wall}{H}
\newcommand{\SpherBuilding}{\Delta}

\newcommand{\NorthPole}{n}
\newcommandx{\SubComplex}[3][1=\SpherBuilding,3={}]{\optionalsubsuper{#1}{#3}{#2}}
\newcommandx{\Equator}[2][2={}]{\SubComplex[#1]{=\pi/2}[#2]}
\newcommandx{\OpenHemisphere}[2][2={}]{\SubComplex[#1]{>\pi/2}[#2]}
\newcommandx{\ClosedHemisphere}[2][2={}]{\SubComplex[#1]{\ge\pi/2}[#2]}
\newcommandx{\Vertical}[2][2={}]{\SubComplex[#1]{\textup{ver}}[#2]}
\newcommandx{\Horizontal}[2][2={}]{\SubComplex[#1]{\textup{hor}}[#2]}
\newcommandx{\Descending}[1]{{#1}^{\downarrow}}
\newcommand{\Min}{^{\textup{min}}}
\newcommand{\Project}[1]{\bar{#1}}

\newcommand{\SomeSet}{T}
\newcommand{\Up}{\nearrow}
\newcommand{\Down}{\searrow}

\newcommand{\Height}{h}
\newcommand{\Depth}{\operatorname{dp}}
\newcommand{\Subdivision}[1]{\mathring{#1}}
\newcommandx{\Link}[2][1={},2={}]{\optionalsubsuper{\lk}{#1}{#2}}
\newcommandx{\DescendingLink}[1][1={}]{\Link[#1][\downarrow]}
\newcommandx{\Star}[2][1={},2={}]{\optionalsubsuper{\st}{#1}{#2}}
\newcommand{\FacePart}{\partial}
\newcommand{\CofacePart}{\delta}
\newcommand{\Boundary}{\partial}
\newcommand{\Gradient}{\nabla}
\newcommand{\InftyGradient}{\Infty\Gradient}
\newcommand{\NullLevel}{\Space_0}

\newcommand{\Roof}[1]{\hat{#1}}
\newcommand{\FaceLattice}{\mathcal{F}}
\newcommand{\Stabilizer}[2]{\operatorname{Stab}_{#1}(#2)}
\newcommand{\Retraction}{\rho}
\newcommand{\WDistance}{\delta}
\newcommand{\WCoDistance}{\WDistance^*}
\newcommand{\PosWDistance}{\WDistance_+}

\newcommand{\Vertices}{\operatorname{vert}}
\newcommand{\Adeles}{\mathbb{A}}

\newcommand{\SystemOfApartments}{\calA}

\title{Finiteness Properties of Chevalley Groups over the\\Laurent Polynomial Ring over a Finite Field}
\author{Stefan Witzel}
\date{}

\begin{document}
\maketitle
%\tableofcontents

% \setlength{\baselineskip}{.55cm}

\begin{abstract}
We show that if $\GroupScheme$ is a Chevalley group of rank $\Dimension$ and $\FiniteField[q][t,t^{-1}]$ is the ring of Laurent polynomials over a finite field, then $\GroupScheme(\FiniteField[q][t,t^{-1}])$ is of type $\TopFin{2\Dimension-1}$. This bound is optimal because it is known -- and we show again -- that the group is not of type $\TopFin{2\Dimension}$.
\end{abstract}

%\newpage

Let $\GlobalField$ be a global field and let $\GroupScheme$ be a connected, noncommutative, absolutely almost simple, isotropic $\GlobalField$-group. Let $\Places$ be a non-empty, finite set of places of $\GlobalField$ and denote by $\Integers$ the set of $\Places$-integers. Let $\SumOfLocalRanks$ be the sum of local ranks of $\GroupScheme$ with respect to the completions of $\GlobalField$ at places in $\Places$. In \cite[page 197]{brown89}, \cite{behr98}, and \cite{buxwor08} the question is raised, whether the following holds:

\begin{RankConjecture}[]
The group $\GroupScheme[\Integers]$ is of type $\TopFin{\SumOfLocalRanks-1}$ but not of type $\TopFin{\SumOfLocalRanks}$.
\end{RankConjecture}

The negative statement was proven by Bux and Wortman in \cite{buxwor07}. Results in favor of the positive statement include \cite{buxwor08} also by Bux and Wortman in which they prove it under the assumption that $\GroupScheme$ has global rank one. If the the $\GlobalField_\Place$-rank is one for all $s \in S$ it also follows from a theorem by Bestvina, Eskin and Wortman according to which $\GroupScheme[\Integers]$ is of type $\TopFin{\abs{S} - 1}$, see \cite{wortman10}. Concerning higher ranks, Bux, Gramlich and the author in \cite{buxgrawit10} proved the conjecture in case $\GroupScheme$ is an $\FiniteField[q]$-group and $\Integers = \FiniteField[q][t]$. The mentioned results were preceded by partial results notably by Abels \cite{abels91}, Abramenko \cite{abramenko87,abramenko96}, Behr \cite{behr98}, Stuhler \cite{stuhler80}, and by Devillers, Gramlich, Mühlherr and the author \cite{devgramue08,grawit09}, see \cite{buxwor07} for a historical overview.

From now on we consider the global field $\GlobalField = \FiniteField[q](t)$. We denote by $\GroupScheme$ a connected, noncommutative, absolutely almost simple $\FiniteField[q]$-group of rank $\Dimension > 0$. The group $\Group \defeq \GroupScheme[{\FiniteField[q][t,t^{-1}]}]$ is $\Places$-arithmetic with $\Places = \{\Class{\Valuation_0},\Class{\Valuation_\infty}\}$ containing the places of the $t$-adic valuation $\Valuation_0$ and the valuation at infinity $\Valuation_\infty$ (see \cite[Chapter~1]{niexin09}). Since we know that the negative part of the Rank Conjecture holds, we see that $\Group$ is not of type $\TopFin{2\Dimension}$. On the other hand the fact that $\GroupScheme[{\FiniteField[q][t]}]$ is of type $\TopFin{\Dimension-1}$ (by \cite{buxgrawit10}) implies that $\Group$ is also of type $\TopFin{\Dimension-1}$ (see Proposition~\ref{prop:increasing_places_preserves_topfin}). The aim of this article is to close the gap by confirming the Rank Conjecture in this case:

\begin{MainTheorem}
Let $\GroupScheme$ be a connected, noncommutative, absolutely almost simple $\FiniteField[q]$-group of rank $\Dimension > 0$.
The group $\GroupScheme[{\FiniteField[q][t,t^{-1}]}]$ is of type $\TopFin{2\Dimension-1}$ but not of type $\TopFin{2\Dimension}$.
\end{MainTheorem}

To our knowledge this is the first time that the full finiteness length (i.e.\ the maximal $\Dimension$ for which the group is of type $\TopFin{n}$, cf.\ \cite[Chapter~7]{geoghegan}) of an $\Places$-arithmetic group of rank $>1$ and with $\abs{\Places} > 1$ is determined. The ``smallest'' case where the result is new is $\SL_3(\FiniteField[q][t,t^{-1}])$. The Main~Theorem was predicted by Abramenko in Conjecture~3 of \cite{abramenko96}.

The completions of $\FiniteField[q](t)$ with respect to $\Valuation_0$ and $\Valuation_\infty$ (see \cite[Chapter~II]{serre79}) are $\FiniteField[q]((t))$ and $\FiniteField[q]((t^{-1}))$ respectively. The group $\Group$ embeds diagonally as a discrete subgroup of $\GroupScheme[{\FiniteField[q]((t))}] \times \GroupScheme[{\FiniteField[q]((t^{-1}))}]$ and in particular acts on the associated Bruhat--Tits building (see \cite{brutit72,brutit84}). But more is true: $\Group$ happens to be a Kac--Moody-group (in the sense of \cite{tits87}) so that there is a twin building that contains the two individual buildings (\cite[Section~10]{buxgrawit10}, see also \cite[Section~6.12]{abrbro}). Using these facts it is not difficult to deduce the Main Theorem from:

\begin{TheoremGeom}
Let $\EuclTwinBuilding$ be an irreducible locally finite Euclidean twin building of dimension $n$. Assume that $\Group$ acts strongly transitively on $\EuclTwinBuilding$ and that the kernel of the action is finite. Then $\Group$ is of type $\TopFin{2n-1}$ but not of type $\TopFin{2n}$.
\end{TheoremGeom}

Recall that a group acts strongly transitively on a twin building if it acts transitively on pairs $(\Chamber,\Apartment)$ where $\Chamber$ is a chamber contained in the twin apartment $\Apartment$ (\cite[Section~6.3]{abrbro}). Using the theory of groups associated to buildings, see Sections~6~to~8 of \cite{abrbro}, we get the following variant of Theorem~\ref{thm:main_theorem_geometric_formulation}:

\begin{TheoremAlg}
Let $\Group$ be a group that admits a twin Tits system $(\Group,B_+,B_-,N,S)$ and set as usual $W = N/(N \intersect B_+)$. If $(W,S)$ is of irreducible affine type and rank $n+1$, $[B_\varepsilon s B_\varepsilon:B_\varepsilon]$ is finite for all $s \in S$ and $\varepsilon \in \{+,-\}$, and $\Intersect_{g \in G}gB_+g^{-1} \intersect \Intersect_{g \in G}gB_-g^{-1}$ is finite, then $\Group$ is of type $\TopFin{2n-1}$ but not of type $\TopFin{2n}$.

This is in particular the case if there is an RGD system $(\Group,(U_\alpha)_{\alpha\in\Phi},T)$ of type $(W,S)$, each $U_\alpha$ is finite, and $\Group_+ \defeq \gen{U_\alpha \mid \alpha \in \Phi}$ has finite index in $\Group$.
\end{TheoremAlg}

The proof of Theorem~\ref{thm:main_theorem_geometric_formulation} is largely a refinement of arguments used in \cite{buxgrawit10}. However since here we have to consider horospheres in a product of two buildings, we cannot reuse the statements from Section~5 of \cite{buxwor08} about horizontal links of Busemann functions and will give a replacement for products of buildings in Section~\ref{sec:horizontal_links}. The Morse function had to be adjusted which as a side-effect led to a more transparent proof thanks to Proposition~\ref{prop:essential_coface_links_dont_subdivide}. Another new complication is resolved from Lemma~\ref{lem:reflect_halve_twin_preserves_height} to Corollary~\ref{cor:apartments_that_contain_chamber_look_good} and explained directly before.

The text is organized as follows: In Section~\ref{sec:spherical_and_euclidean_buildings}, we collect facts about spherical and Euclidean buildings that we will need later on. In Section~\ref{sec:horizontal_links} we prove the mentioned statements about horizontal links of Busemann functions. They may be of independent interest for proving the rank conjecture for other $\Places$-arithmetic groups with $\abs{\Places} > 1$. The main result of that section -- and a central ingredient of this article -- is Proposition~\ref{prop:bound_on_moves}.

Sections~\ref{sec:zonotopes} and \ref{sec:euclidean_twin_buildings} provide tools we need from convex geometry and the theory of twin buildings respectively. The actual proof of Theorem~\ref{thm:main_theorem_geometric_formulation} starts in Section~\ref{sec:height_and_gradient} with an outline.

\vspace{.5\baselineskip}

The work for this article was supported by DFG and Studienstiftung des deutschen Volkes. During the research for this article the author enjoyed the hospitality of SFB~701 of the Universität~Bielefeld as well as of Mathematisches Forschungsinstitut Oberwolfach. The author wants to thank Jan Essert and Petra Schwer for helpful discussions and Markus-Ludwig Wermer and Andreas Mars for comments on the manuscript. The author is especially indebted to Ralf Gramlich and Kai-Uwe Bux for many helpful comments and critical questions during the research for this article.

% ***************
% *** SECTION ***
% ***************

\section{Spherical and Euclidean buildings}
\label{sec:products_of_euclidean_buildings}
\label{sec:spherical_subcomplexes}
\label{sec:spherical_and_euclidean_buildings}

In this section we recall some known properties about spherical and Euclidean buildings and collect additional facts that we will need in later sections. Our general reference for buildings is \cite{abrbro}. We regard buildings as cell complexes equipped with a metric as is described in the separate sections below. The metrics have bounded curvature in the sense that they satisfy the CAT($\kappa$)-inequality (see \cite[Part~II]{brihae}) for $\kappa = 1$ (if the building is spherical) respectively $\kappa = 0$ (if the building is Euclidean).

\subsection*{Spherical buildings}

Let $\SpherBuilding$ be a thick spherical building regarded as a simplicial complex. The apartments of $\SpherBuilding$ are tessellated unit spheres which we equip with the angular metric (see \cite[Chapter~I.2]{brihae}). The metrics on the apartments fit together to give a metric $\SpherDistance$ on $\SpherBuilding$ that is CAT($1$) (see \cite[Example~12.39]{abrbro}).

Recall that $\SpherBuilding$ decomposes as a spherical join $\Join_i \SpherBuilding_i$ of irreducible spherical buildings.

We denote by $\typ \SpherBuilding$ the Coxeter diagram of $\SpherBuilding$ and if $\Cell$ is a simplex of $\SpherBuilding$, then $\typ \Cell$ is the subdiagram spanned by the types of vertices of $\Cell$. Note that the type of $\Link \Cell$ is the diagram $\typ \SpherBuilding \setminus \typ \Cell$ obtained from $\typ \SpherBuilding$ by removing nodes and edges that meet $\typ \Cell$. Recall also that the connected components of $\typ\SpherBuilding$ are $\typ\SpherBuilding_i$ for the irreducible join factors $\SpherBuilding_i$.

We will encounter spherical buildings as links of an affine building. We will filter the affine building and it will be important that the descending parts of the links with respect to the filtration be highly connected. So in what follows we provide a supply of highly connected subcomplexes of spherical buildings. For us \emph{$\Dimension$-spherical} means $\Dimension$-dimensional and $(\Dimension-1)$-connected and \emph{properly $\Dimension$-spherical} means $\Dimension$-spherical and not contractible.

Let $\NorthPole \in \SpherBuilding$ be an arbitrary point of which we think as the ``north pole''. Depending on $\NorthPole$ there are the following subcomplexes of $\SpherBuilding$: The \emph{equator complex} $\Equator{\SpherBuilding}$ of simplices $\Cell \subseteq \Delta$ such that $\SpherDistance(\NorthPole,\Point) = \pi/2$ for every $\Point \in \Cell$. The \emph{closed hemisphere complex} $\ClosedHemisphere{\SpherBuilding}$ and the \emph{open hemisphere complex} $\OpenHemisphere{\SpherBuilding}$ are defined analogously.

The \emph{horizontal part} $\Horizontal{\SpherBuilding}$ is the join of join factors $\SpherBuilding_i$ that are contained in $\Equator{\SpherBuilding}$ and the \emph{vertical part} $\Vertical{\SpherBuilding}$ is the join of the remaining join factors so that there is obviously a decomposition
\begin{equation}
\label{eq:horizontal_vertical_decomposition}
\SpherBuilding = \Horizontal{\SpherBuilding} * \Vertical{\SpherBuilding} \text{ .}
\end{equation}
Note that being horizontal (contained in the horizontal part) is more restrictive than being equatorial (contained in the equator complex). This is illustrated by:

\begin{obs}
\label{obs:equator_decomposition}
Let $\SpherBuilding = \Join_i \Lambda_i$ be a join of (not necessarily irreducible) spherical buildings $\Lambda_i$ and let $\NorthPole \in \SpherBuilding$. For every $\Lambda_i$ not contained in $\Horizontal\SpherBuilding$ let $\NorthPole_i$ be the projection of $\NorthPole$ to $\Lambda_i$. A simplex $\Cell = \Join_i \Cell_i$ is equatorial if and only if $\Cell_i$ is equatorial in $\Lambda_i$ with respect to $\NorthPole_i$ whenever $\NorthPole_i$ exists.\qed
\end{obs}

Each $\Lambda_i$ is $\pi$-convex (as defined in \cite[Chapter~I.1]{brihae}) and it is contained in $\Horizontal\SpherBuilding$ if and only if it has distance $\pi/2$ from $\NorthPole$. The projection mentioned in the observation is the one given by:

\begin{lem}
\label{lem:spherical_projection}
If $\ConvexSet \subseteq \SpherBuilding$ is a $\pi$-convex set and $\SpherPoint \in \SpherBuilding$ satisfies $d(\SpherPoint,\ConvexSet) < \pi/2$ then there is a unique point $\Proj[\ConvexSet] \SpherPoint$ in $\ConvexSet$ that is closest to $\SpherPoint$. Moreover the angle $\angle_{\Proj[\ConvexSet]\SpherPoint}(\SpherPoint,\ConvexSetPoint)$ is at least $\pi/2$ for every $\ConvexSetPoint \in \ConvexSet$.
\end{lem}

\begin{proof}
This is proven as Proposition~II.2.4~(1) in \cite{brihae}, see also Exercise~II.2.6~(1).
\end{proof}

Schulz has investigated hemisphere complexes in his thesis \cite{schulz}, see also \cite{schulz10}. His main result is:

\begin{thm}
\label{thm:hemisphere_complexes}
The closed hemisphere complex is properly $(\dim \SpherBuilding)$-spherical. The open hemisphere complex is properly $(\dim \Vertical{\SpherBuilding})$-spherical.
\end{thm}

The statement for closed hemisphere complexes is immediate from Schulz' more general result that the subcomplex supported by any non-empty coconvex set of $\SpherBuilding$ is $(\dim \SpherBuilding)$-spherical (the subcomplex supported by a subset is the subcomplex of cells contained in that subset). The proof of this result can be extended to give \cite[Proposition~3.3]{buxgrawit10}:

\begin{prop}
\label{prop:apartmentwise_coconvex_complexes}
Let $\Chamber$ be a chamber of $\SpherBuilding$ and let $\ConvexSet \subseteq \SpherBuilding$ be such that for every apartment $\Apartment$ that contains $\Chamber$ the intersection $\ConvexSet \intersect \Apartment$ is a proper, convex, open subset of $\Apartment$. Then the subcomplex supported by $\SpherBuilding \setminus \ConvexSet$ is $(\dim \SpherBuilding)$-spherical.
\end{prop}

To prove this kind of statements, \cite[Proposition~3.2]{buxgrawit10} is a useful tool. We only need the following special case:

\begin{prop}
\label{prop:retract_set_onto_subcomplex}
If $\SpherBuilding$ is a spherical building and $\ConvexSet$ is an open subset that meets every chamber in a convex set, then $\SpherBuilding \setminus \ConvexSet$ deformation retracts onto the subcomplex that it supports.
\end{prop}

Theorem~\ref{thm:hemisphere_complexes} shows that it is crucial to understand whether a simplex is horizontal or not. For this we have the following criterion, see \cite[Lemma~4.2]{buxwor08}:

\begin{lem}
\label{lem:geometric_criterion}
A simplex $\Cell \subseteq \SpherBuilding$ lies in $\Horizontal\SpherBuilding$ if and only if $\SpherDistance(\Cell,\Vertex) = \pi/2$ for every non-equatorial vertex $\Vertex$ adjacent to $\Cell$.
\end{lem}

This can also be formulated in a diagram-theoretic way as follows:

\begin{lem}
\label{lem:diagram_criterion}
Let $\Cell \subseteq \SpherBuilding$ be a simplex and let $\Chamber$ be a chamber that contains $\Cell$. Then $\Cell$ lies in $\Horizontal\SpherBuilding$ if and only if
every vertex $\Vertex$ of $\Chamber$, such that there is a path from $\typ \Vertex$ to $\typ \Cell$, is equatorial.
%no vertex of $\typ \Cell$ can be joined to $\typ\Vertex$ for a non-equatorial vertex $\Vertex$ of $\Chamber$.
\end{lem}

\begin{proof}
Clearly every horizontal vertex is equatorial. So what remains to be shown is that if $\SpherBuilding_i$ is an irreducible component of $\SpherBuilding$ that is contained in $\Vertical\SpherBuilding$, then $\SpherBuilding_i \intersect C$ contains a non-equatorial vertex. But $\SpherBuilding_i \intersect C$ is a chamber of $\SpherBuilding_i$ which can only be equatorial if all of $\SpherBuilding_i$ is.
\end{proof}

\subsection*{Euclidean buildings}

Let $\Space$ be a building of affine type regarded as a polysimplicial complex. Each apartment is a tessellated Euclidean space and the metrics on the apartments fit together to a metric on $\Space$ that is CAT($0$) (see \cite[Section~11.2]{abrbro}). We will denote this metric by $\EuclDistance$. Unless explicitly stated otherwise, we consider the complete system of apartments on $\Space$ (see \cite[Section~4.5]{abrbro}).

Let $\Point \in \Space$ be any point. An isometry from a closed interval into $\Space$ is a \emph{geodesic}, see \cite[Chapter~I.1]{brihae} for related definitions. A geodesic $\Geodesic$ that issues at $\Point$ defines a \emph{direction} $\Geodesic_\Point$ which is the set of geodesics that coincide with $\Geodesic$ on an initial interval (see \cite[Chapter~II.3]{brihae}). The \emph{link of a point $\Point \in \Space$} is the set $\Link[\Space] \Point$ (or just $\Link \Point$ if misunderstandings are unlikely) of directions of geodesics in $\Space$ that issue at $\Point$. The link is a CAT($1$)-space via the angular metric
\[
d_{\Link\Point}(\Geodesic_\Point,\Another\Geodesic_\Point) = \angle_{\Point}(\Geodesic,\Another\Geodesic)
\]
(see \cite[Chapter~II.3]{brihae}). To avoid confusion we will often prefer the angle notation on the right even for directions.

If $\Cell \subseteq \Space$ is a (non-empty) cell and $\Geodesic$ is a geodesic that issues at an interior point $\Point$ of $\Cell$, we say that $\Geodesic_\Point$ is \emph{perpendicular} to $\Cell$ if $\angle_\Point(\Geodesic,\Another\Point) = \pi/2$ for every $\Another\Point \in \Cell \setminus \Point$. The links of two interior points of $\Cell$ are naturally isometric and the subspace of directions perpendicular to $\Cell$ of any of them is the \emph{link of $\Cell$}, denoted $\Link \Cell$.

The link of a cell $\Cell$ carries a simplicial cell structure, namely if $\BigCell$ is a coface of $\Cell$ then the directions that point into $\BigCell$ form a simplex $\BigCell \direction \Cell$ of $\Link \Cell$. In fact $\Link \Cell$ (the simplicial complex equipped with the CAT($1$)-metric) is a spherical building. The notation $\BigCell \direction \Cell$ is non-standard, in a simplicial complex it is common to write $\BigCell \setminus \Cell$ instead (which is reminiscent of abstract simplicial complexes) and we will follow this custom whenever we talk about simplicial complexes.

The link of a point $\Point$ decomposes as a spherical join
\begin{equation}
\label{eq:link_of_point_decomposition}
\Link \Point = \Boundary_\Point \Cell * \Link \Cell
\end{equation}
where $\Cell$ is the carrier of $\Point$ (i.e.\ the smallest cell that contains it) and $\Boundary_\Point \Cell$ is the boundary of $\Cell$ regarded as the space of directions $\Geodesic_\Point$ where $\Geodesic$ is a geodesic segment from $\Point$ to a point of $\Cell$. The subscript takes account of the fact that the angular metric on $\Boundary \Cell$ depends on $\Point$.

What we have just said about links of points and cells holds in the same way for spherical buildings.

The (closed) \emph{star $\Star\Cell$} of a cell $\Cell$ is the subcomplex of $\Space$ consisting of cells that contain $\Cell$ and their faces. Recall that whenever $\Chamber$ is a chamber, there is a unique chamber in $\Star\Cell$ that has shortest Weyl-distance to $\Chamber$, called the \emph{projection of $\Chamber$ onto $\Star\Cell$} and denoted $\Proj[\Star\Cell]\Chamber$ (\cite[Section~4.9]{abrbro}). If $\BigCell$ is an arbitrary cell we define $\Proj[\Star\Cell]\BigCell \defeq \Intersect_{\Chamber \ge \BigCell} \Proj[\Star\Cell]\Chamber$.

Recall also the concept of retractions (\cite[Section~4.4]{abrbro}): Given a chamber $\Chamber$ that is contained in an apartment $\Apartment$, there is a map $\Retraction \colon \Space \to \Apartment$ that fixes $\Apartment$, preserves distance from $\Chamber$ (that is, it preserves Weyl-distance from $\Chamber$ as well as metric distance from every point of $\Chamber$), and is contracting otherwise. It takes a chamber $\AltChamber$ to the unique chamber in $\Apartment$ that has distance $\WDistance(\Chamber,\AltChamber)$ to $\Chamber$ where $\WDistance$ denotes Weyl-distance. The map $\Retraction$ is called the \emph{retraction onto $\Apartment$ centered at $\Chamber$}.

Now we turn to the asymptotic structure of $\Space$ (\cite[Chapters~II.8,~II.9]{brihae}). Two geodesic rays in $\Space$ are \emph{asymptotic} if they have bounded distance. Every geodesic ray $\Ray$ defines a Busemann function
\[
\Busemann(\Point) = \lim_{t \to \infty} t - \EuclDistance(\Ray(t),\Point) \text{ .}
\]
Beware that our Busemann functions have reversed sign compared to the definition in \cite{brihae}. If two rays define the same Busemann function then they are asymptotic. Conversely the Busemann functions $\Busemann_1$, $\Busemann_2$ defined by two asymptotic rays may differ by a constant, i.e.\ $\Busemann_1 = \Busemann_2 + \Constant$. A \emph{point at infinity} is the class $\Infty\Ray$ of rays asymptotic to a given ray $\Ray$ or, equivalently, the class $\Infty\Busemann$ of Busemann functions that differ from a given Busemann function $\Busemann$ by a constant. The \emph{visual boundary} $\Infty\Space$ consists of all points at infinity. It becomes a CAT(1)-space via the Tits metric
\[
d_{\Infty\Space}(\Infty\Ray,\Infty{\Another\Ray}) = \angle(\Ray,\Another\Ray)
\]
(see \cite[Chapter~II.9]{brihae}). Let $\Vertex$ be a special vertex (\cite[Definition~10.18]{abrbro}), let $\Chamber \ge \Vertex$ be a chamber, and let $\Apartment$ be an apartment that contains $\Chamber$; the subset of $\Apartment$ covered by geodesic rays in $\Apartment$ that issue at $\Vertex$ and have an initial segment contained in $\Chamber$ is called a \emph{sector} of $\Apartment$. Taking visual boundaries of sectors to be maximal simplices turns $\Infty\Space$ into a simplicial complex that is in fact a spherical building, the \emph{building at infinity of $\Space$} (with respect to the complete system of apartments), see \cite[Theorem~11.79]{abrbro}.

For every point $\Point \in \Space$ and every point at infinity $\PointAtInfty \in \Infty\Space$, there is a unique geodesic ray $\Ray$ that issues at $\Point$ and defines $\PointAtInfty$, i.e.\ $\PointAtInfty = \Infty\Ray$ (see \cite[Proposition~II.8.2]{brihae}). This means that a point at infinity $\PointAtInfty$ defines a direction $\PointAtInfty_\Point \in \Link \Point$ at any point $\Point$, namely $\PointAtInfty_\Point \defeq \Ray_\Point$ where $\Ray$ issues at $\Point$ and defines $\PointAtInfty$. Using this we get a projection from the asymptotic structure to the local structure:

\begin{obs}
\label{obs:asymptotic_to_local_is_cellular}
For every point $\Point \in \Infty\Space$ the map $\Infty\Space \to \Link\Point$ that takes $\PointAtInfty$ to $\PointAtInfty_\Point$ maps simplices into simplices.\qed
\end{obs}

For the following statement it is crucial that $\Space$ is Euclidean (in hyperbolic space a Busemann function cannot be constant on a non-trivial geodesic).

\begin{obs}
\label{obs:busemann_angle_criterion}
Let $\PointAtInfty \in \Infty\Space$ be a point at infinity and let $\Point \in \Space$ be a point. Assume that the Busemann function $\Busemann$ defines $\PointAtInfty$. If $\Geodesic$ issues at $\Point$ then $\Busemann$ is constant on an initial interval $\Geodesic$ (i.e.\ $\Busemann \circ \Geodesic$ is constant on an initial interval) if and only if $\angle_\Point(\Geodesic_\Point,\PointAtInfty_\Point) = \pi/2$. And $\Busemann$ is increasing (decreasing) on $\Geodesic$ if and only if $\angle_\Point(\Geodesic_\Point,\PointAtInfty_\Point) < \pi/2$ ($> \pi/2$).\qed
\end{obs}

The Euclidean building $\Space$ decomposes as a direct product $\prod_i \EuclBuilding_i$ of irreducible Euclidean buildings. The visual boundary $\Infty\Space$ is the spherical join $\Join_i \Infty\EuclBuilding_i$ of the individual visual boundaries. Similarly the link $\Link[\Space] \Cell$ of a cell $\Cell = \prod_i \Cell_i \in \Space$ is the spherical join  $\Join_i \Link[\EuclBuilding_i] \Cell_i$ of the links of its components.

At every point $\Point = (\Point_i)_i \in \Space$ we have embeddings
\begin{eqnarray*}
\iota_{i,\Point} \colon \EuclBuilding_i &\to& \Space\\
\AltPoint_i &\mapsto& (\AltPoint_i,(\Point_j)_{j \ne i}) \text{ .}
\end{eqnarray*}

\begin{obs}
\label{obs:horizontal_perpendicular}
Let $\PointAtInfty \in \Infty\Space$ and $\Point \in \Space$. Let $\Ray$ be the geodesic ray that issues at $\Point$ and defines $\PointAtInfty$ and let $\Busemann$ be the Busemann function it defines. Identify every $X_j$ with $\iota_{j,\Point}(X_j)$. Let $i$ be any index. The following are equivalent:
\begin{enumerate}
\item $\Ray$ is perpendicular to $\EuclBuilding_i$, i.e.\ $\angle_\Point(\Ray,\Another\Point) = \pi/2$ for every $\Another\Point \in \EuclBuilding_i$.
\item $\Ray$ is contained in $\prod_{j \ne i} \EuclBuilding_j$.
\item $\Busemann$ is constant on $X_i$.
\item $\PointAtInfty \in \Join_{j \ne i} \Infty\EuclBuilding_j$.\qed
\end{enumerate}
\end{obs}

\begin{rem}
Since the last statement is independent of the basepoint $\Point$, so are the other.
\end{rem}

Let $\Proj[i]$ denote the projection of $\Space$ onto $\EuclBuilding_i$. If $\Ray$ is a geodesic ray in $\Space$ then $\Proj[i] \circ \Ray$ is a reparametrized geodesic ray, i.e.\ there is a geodesic ray $\Another\Ray$ in $\EuclBuilding_i$ and an $\Parameter \in [0,1]$ such that $\Proj[i] \circ \Ray(t) = \Another\Ray(\Parameter \cdot t)$. It is clear from the last observation that $r = 0$ if and only if $\Ray$ is perpendicular to $\EuclBuilding_i$. If this is not the case, then $\Another\Ray$ can be recovered from $\Proj[i] \circ \Ray$ and we denote it by $\Proj[i] \Ray$.

At infinity this translates as follows: $\Infty\EuclBuilding_i$ is $\pi$-convex in $\Infty\Space$. For a point $\PointAtInfty \in \Space$ either $\PointAtInfty \in \Join_{j \ne i} \Infty\EuclBuilding_j$ or $d(\PointAtInfty,\Infty\EuclBuilding_i) < \pi/2$. In the latter case $\PointAtInfty$ projects to a unique projection point $\Proj[i]\PointAtInfty$ in $\Infty\EuclBuilding_i$ by Lemma~\ref{lem:spherical_projection}. The two projections we just defined are compatible in the sense that $\Proj[i](\Infty\Ray) = \Infty{(\Proj[i]\Ray)}$.

\begin{obs}
\label{obs:perp_factorwise}
Let $\Ray$ be a geodesic ray in $\Space$ that is perpendicular to $\Cell = \prod_i \Cell_i$. If $\BigCell = \prod_i \BigCell_i$ is a coface of $\Cell$, then $\Ray$ is perpendicular to $\BigCell$ if and only if $\Proj[i]\Ray$ is perpendicular to $\BigCell_i$ for every $i$ for which $\Ray$ is not perpendicular to $\EuclBuilding_i$.
\end{obs}

\begin{proof}
Let $\Point$ be the interior point of $\Cell$ where $\Ray$ issues. The spherical building $\Link\Cell$ decomposes as a spherical join $\Join_i \Link \Cell_i$. If $\Ray$ is perpendicular to $\EuclBuilding_i$, then clearly $\Ray_\Point$ has distance $\pi$ from $\Link \Cell_i$. Otherwise the projection of $\pr_i(\Ray_\Point)$ onto $\Link \Cell_i$ is the same as $(\pr_i \Ray)_\Point$. So the statement follows from Observation~\ref{obs:equator_decomposition}.
\end{proof}

% ***************
% *** SECTION ***
% ***************

\section{Horizontal links}
\label{sec:horizontal_links}

The purpose of this section is to generalize Section~5 of \cite{buxwor08} to Euclidean buildings that are not necessarily irreducible.
Though Bux and Wortman do deal with products of buildings (in Section~7 of the same article), they can avoid the generality we are aiming for, because their Morse function is a linear combination of functions each defined on one factor only, while ours will properly be a function of the product.

The superficial problem in adapting the results of \cite{buxwor08} is that a product of buildings is not simplicial. In particular it cannot be a flag complex. Thus it is not immediately clear how to translate the arguments that are carried out in the Euclidean building. Though this is certainly possible, we prefer to take a different approach: namely to deduce all that we need from the characterizing property of the fundamental object of study, the face $\BigCell\Min$ of a cell $\BigCell$; and to carry out the argument inside the links, which are spherical buildings.

Throughout the section let $\Space = \prod_i\EuclBuilding_i$ be a finite product of irreducible Euclidean buildings. Even if $\Space$ is not a flag complex, it has the following property reminiscent of flag complexes:

\begin{obs}
\label{obs:polyflag}
If $\Cell^1,\ldots,\Cell^k$ are cells in a product of flag-complexes and for $1 \le l < m \le k$ the cell $\Cell^l \vee \Cell^m$ exists, then $\Cell^1 \vee \cdots \vee \Cell^k$ exists.
\end{obs}

\begin{proof}
Write $\Cell^l = \prod_i \Cell^l_i$. Then $\Cell^l \vee \Cell^m$ exists if and only if $\Cell^l_i \vee \Cell^m_i$ exist for every $i$. The statement is thus translated to a family of statements, one for each factor, that hold because the factors are flag-complexes.
\end{proof}

Let $\Busemann$ be a Busemann function on $\Space$. A cell on which $\Busemann$ is constant is called \emph{horizontal}. If $\Cell$ is horizontal and $\Ray$ is a ray that defines $\Busemann$ and issues at a point of $\Cell$, then $\Ray$ is perpendicular to $\Cell$ by Observation~\ref{obs:busemann_angle_criterion} so that it defines a point $\NorthPole$ in $\Link \Cell$. This point shall be our north pole and the notions from Section~\ref{sec:spherical_subcomplexes} carry over accordingly. In particular the \emph{horizontal link} $\Horizontal\Link\Cell$ is the join of all join factors of $\Link \Cell$ that are perpendicular to $\NorthPole$.

We write
\[
\BigCell \horizontal \Cell \quad \text{if} \quad \BigCell \direction \Cell \subseteq \Horizontal\Link \Cell
\]
and say for short that ``$\BigCell$ lies in the horizontal link of $\Cell$''. Note that this in particular requires $\BigCell$ to be horizontal (but is a stronger condition). Observation~\ref{obs:horizontal_perpendicular} tells us that lying in the horizontal link actually does not depend on $\Busemann$ but only on the point at infinity that it defines. If we want to emphasize the point at infinity $\PointAtInfty$ with respect to which $\BigCell$ lies in the horizontal link of $\Cell$ then we write $\BigCell \horizontal_\PointAtInfty \Cell$.

The next two observations deal with the interaction of $\PointAtInfty$ and its projections onto the factors of $\Infty\Space$.

\begin{obs}
\label{obs:horizontal_iff_horizontal_on_factors}
Let $\BigCell = \prod_i \BigCell_i$ and $\Cell = \prod_i \Cell_i$ be non-empty horizontal cells. Let $I$ be the set of indices $i$ such that $d(\PointAtInfty,\Infty\EuclBuilding_i) \ne \pi/2$.
Then
\[
\BigCell \horizontal_{\PointAtInfty} \Cell \quad \text{if and only if} \quad \BigCell_i \horizontal_{\Proj[i]\PointAtInfty} \Cell_i \text{ for every } i \in I\text{ .}
\]
In other words
\[
\Horizontal\Link\Cell = (\Join_{i \in I} \Horizontal\Link \Cell_i) * (\Join_{i \nin I} \Link \Cell_i)
\]
where the north pole of $\Link \Cell_i$ is given by $\Proj[i]\PointAtInfty$.
\end{obs}

\begin{proof}
Recall that $\Link \Cell$ decomposes as $\Join_i \Link \Cell_i$. So if $\SpherBuilding'$ is an irreducible join factor of $\Link \Cell$, then it is clearly a join factor of some $\Link \Cell_i$. By Observation~\ref{obs:perp_factorwise} either $d(\PointAtInfty,\Infty\EuclBuilding_i) = \pi/2$ and $\SpherBuilding'$ is automatically equatorial or $\SpherBuilding'$ is equatorial in $\Link \Cell$ if and only if it is equatorial in $\Link \Cell_i$ with respect to $\Proj[i]\PointAtInfty$.
\end{proof}

For each factor we have:

\begin{lem}
\label{lem:cohorizontal_faces_meet}
Let $\EuclBuilding$ be an irreducible Euclidean building and $\PointAtInfty \in \Infty\EuclBuilding$. If $\BigCell \horizontal_\PointAtInfty \Cell_1$ and $\BigCell \horizontal_\PointAtInfty \Cell_2$ then $\Cell_1 \intersect \Cell_2 \ne \emptyset$.
\end{lem}

\begin{proof}
Let $\Busemann$ be a Busemann function that defines $\PointAtInfty$. Let $\Chamber$ be a chamber that contains $\BigCell$ and let $\Vertex$ be a vertex of $\Chamber$ with $\Busemann(v) \ne \Busemann(\BigCell)$. We take the quotient of $\Chamber$ modulo directions in $\Cell_1$ and $\Cell_2$. The images of $\BigCell$, $\Vertex$, $\Cell_1$, and $\Cell_2$ under this projection are denoted $\Project\BigCell$, $\Project\Vertex$, $\Project\Cell_1$, and $\Project\Cell_2$ respectively. If $\Cell_1$ and $\Cell_2$ did not meet, then $\Project\Cell_1$ and $\Project\Cell_2$ would be distinct points. In any case $\Project\Vertex$ is distinct from both. By Lemma~\ref{lem:geometric_criterion} we would have $\angle_{\Project\Cell_1}(\Project\Cell_2,\Project\Vertex) = \angle_{\Project\Cell_2}(\Project\Cell_1,\Project\Vertex) = \pi/2$ which is impossible.
\end{proof}

For the rest of the section all horizontal links are taken with respect to a fixed Busemann function $\Busemann$.
We say that $\Busemann$ (or $\Infty\Busemann$) is \emph{in general position} if it is not constant on any factor of $\Space$. Combining Observation~\ref{obs:horizontal_iff_horizontal_on_factors} and Lemma~\ref{lem:cohorizontal_faces_meet} we see:

\begin{obs}
\label{obs:general_position_cells_meet}
Assume that $\Busemann$ is in general position. If $\BigCell \horizontal \Cell_1$ and $\BigCell \horizontal \Cell_2$ then $\Cell_1 \intersect \Cell_2 \ne \emptyset$.\qed
\end{obs}

The assumption that $\Busemann$ be in general position is crucial as can be seen in the most elementary case:

\begin{exmpl}
\label{exmpl:cohorizontal_faces_dont_meet}
Consider the product $\EuclBuilding_1 \times \EuclBuilding_2$ of two buildings of type $\tilde{A}_1$. Let $\Busemann$ be such that $\Infty\Busemann \in \Infty\EuclBuilding_1$. Let $\Vertex_1 \in \EuclBuilding_1$ be a vertex and $\Chamber_2 \subseteq \EuclBuilding_2$ be a chamber with vertices $\Vertex_2$ and $\AltVertex_2$. The links of $\Vertex_1 \times \Vertex_2$ and of $\Vertex_1 \times \AltVertex_2$ are of type $A_1 * A_1$ and $\Vertex_1 \times \Chamber_2$ lies in the horizontal link of both.
\end{exmpl}

We are now ready to state a technical tool that we will use throughout the section. We will give two proofs at the end of the section.

\begin{prop}
\label{prop:horizontal_properties}
The relation $\horizontal$ (that is $\horizontal_{\Infty\Busemann}$) has the following properties:
\begin{enumerate}
\item If $\BigCell \horizontal \Cell$ and $\BigCell \ge \Another\BigCell \ge \Cell$ then $\Another\BigCell \horizontal \Cell$.\label{item:faces}
\item If $\BigCell \horizontal \Cell$ and $\BigCell \vee \Another\Cell$ exists and is horizontal then $\BigCell \vee \Another\Cell \horizontal \Cell \vee \Another\Cell$. In particular if $\BigCell \horizontal \Cell$ and $\BigCell \ge \Another\Cell \ge \Cell$ then $\BigCell \horizontal \Another\Cell$. \label{item:join}
\item If $\BigCell \horizontal \Another\Cell$ and $\Another\Cell \horizontal \Cell$ then $\BigCell \horizontal \Cell$, i.e.\ $\horizontal$ is transitive.\label{item:transitivity}
\item If $\BigCell \horizontal \Cell_1$ and $\BigCell \horizontal \Cell_2$ and $\Cell_1 \intersect \Cell_2 \ne \emptyset$ then $\BigCell \horizontal \Cell_1 \intersect \Cell_2$.\label{item:non_empty_meet}
\end{enumerate}
\end{prop}

A key observation in \cite{buxwor08} is that for every horizontal cell $\BigCell$, among its faces $\Cell$ with $\BigCell \horizontal \Cell$ there is a minimal one provided $\Space$ is irreducible. Observation~\ref{obs:general_position_cells_meet} allows us to replace the irreducibility assumption by the assumption that $\Busemann$ be in general position:

\begin{lem}
\label{lem:tau_min}
Assume that $\Busemann$ is in general position. Let $\BigCell$ be a horizontal cell of $\Space$. The set of $\Cell \le \BigCell$ such that $\BigCell \horizontal \Cell$ is an interval, i.e.\ it contains a minimal element $\BigCell\Min$ and
\[
\BigCell \horizontal \Cell \quad \text{if and only if} \quad \BigCell\Min \le \Cell \le \BigCell \text{ .}
\]
\end{lem}

At the risk of repetition note that in particular $\BigCell \horizontal \BigCell\Min$.

\begin{proof}
Let $\SomeSet \defeq \{\Cell \le \BigCell \mid \BigCell \horizontal \Cell\}$, which is finite. If $\Cell_1$ and $\Cell_2$ are in $\SomeSet$, then since $\Busemann$ is in general position, Observation~\ref{obs:general_position_cells_meet} implies that $\Cell_1 \intersect \Cell_2 \ne \emptyset$. So by Proposition~\ref{prop:horizontal_properties}~\eqref{item:non_empty_meet} $\Cell_1 \intersect \Cell_2 \in \SomeSet$. Hence there is a minimal element $\BigCell\Min$, namely the intersection of all elements of $\SomeSet$. If $\Another\Cell$ satisfies $\BigCell\Min \le \Another\Cell \le \BigCell$, then $\Another\Cell \in \SomeSet$ by Proposition~\ref{prop:horizontal_properties}~\eqref{item:join}.
\end{proof}

To see what can go amiss if $\Busemann$ is not in general position we need a slightly bigger example than Example~\ref{exmpl:cohorizontal_faces_dont_meet}:
\begin{exmpl}
Let $\EuclBuilding = \EuclBuilding_1 \times \EuclBuilding_2$ where the first factor is of type $\tilde{A}_1$ and the second is of type $\tilde{A_2}$ (both thin if you want). The factor $X_2$ has three parallelity classes of edges. Let $\Busemann$ be a Busemann function that is perpendicular to $\EuclBuilding_1$ and to one class of edges in $\EuclBuilding_2$. Consider a square $\BigCell$ that is horizontal. Its two edges in the $\EuclBuilding_2$-factor have a link of type $A_1 * A_1$ and the square lies in the horizontal link of each of them. Hence if there were to be a $\BigCell\Min$ it would have to be the empty simplex. However the vertices have a link of type $A_1 * A_2$ and the square does not lie in their horizontal link.
\end{exmpl}

To understand this example note that if a Busemann function is constant on some factor, then in this factor $\BigCell\Min = \emptyset$ for all cells $\BigCell$. But being empty does not behave well with respect to taking products: a product is empty if one of the factors is empty, not if all of the factors are empty. In other words the face lattice of a product of simplices is not the product of the face lattices of the simplices. But the face lattice of a product of simplices without the bottom element is the product of the face lattices of the simplices without the bottom elements: $\FaceLattice(\prod_i \Cell_i)_{> \emptyset} = \prod_i \FaceLattice(\Cell_i)_{> \emptyset}$.

Lemma~\ref{lem:tau_min} generalizes \cite[Lemma~5.2]{buxwor08} except for the explicit description in terms of orthogonal projections. We will see that transitivity of $\horizontal$ suffices to replace the explicit description. For the rest of the section we assume that $\Busemann$ is in general position.

We define \emph{going up} by
\[
\Cell \Up \BigCell \quad \text{if} \quad \BigCell\Min = \Cell \ne \BigCell
\]
and \emph{going down} by
\[
\BigCell \Down \Cell \quad \text{if} \quad \Cell \lneq \BigCell \text{ but not }\BigCell \horizontal \Cell\text{ .}
\]
A \emph{move} is either going up or going down. The main result of this section is:

\begin{prop}
\label{prop:bound_on_moves}
There is a bound on the length of sequences of moves that only depends on the types of the $\EuclBuilding_i$. In particular no sequence of moves enters a cycle.
\end{prop}

The results in \cite{buxwor08} for which the arguments do not apply analogously are Observation~5.3 and the Lemmas~5.10 and 5.13. They correspond to Observation~\ref{obs:min_min}, Lemma~\ref{lem:min_in_vertical_link}, and Lemma~\ref{lem:cell1_bigcell2_min} below. For the convenience of the reader we also give proofs of the statements that can be easily adapted from those in \cite{buxwor08}.

A good starting point is of course:

\begin{obs}
\label{obs:no_up_down_cycle}
There do not exist cells $\Cell$, $\BigCell$ such that $\Cell \Up \BigCell$ and $\BigCell \Down \Cell$.
\end{obs}

\begin{proof}
If $\Cell \Up \BigCell$, then in particular $\BigCell \horizontal \Cell$ which contradicts $\BigCell \Down \Cell$.
\end{proof}

We come to the first example of how transitivity of $\horizontal$ replaces the explicit description of $\BigCell\Min$:

\begin{obs}
\label{obs:min_min}
If $\BigCell \horizontal \Cell$, then $\Cell\Min = \BigCell\Min$. In particular $(\BigCell\Min)\Min = \BigCell\Min$.
\end{obs}

\begin{proof}
We have $\BigCell \horizontal \BigCell\Min$ and $\BigCell \ge \Cell \ge \BigCell\Min$ so by Proposition~\ref{prop:horizontal_properties}~\eqref{item:faces} $\Cell \horizontal \BigCell\Min$, i.e.\ $\Cell\Min \le \BigCell\Min$. Conversely $\BigCell \horizontal \Cell \horizontal \Cell\Min$ so by Proposition~\ref{prop:horizontal_properties}~\eqref{item:transitivity} $\BigCell \horizontal \Cell\Min$, i.e.\ $\BigCell\Min \le \Cell\Min$.
\end{proof}

We call a cell $\Cell$ \emph{essential} if $\Cell\Min = \Cell$.

\begin{obs}
\label{obs:essential_either_or}
If $\Cell$ is essential and $\BigCell \gneq \Cell$ is a proper horizontal coface, then either $\Cell \Up \BigCell$ or $\BigCell \Down \Cell$.
\end{obs}

\begin{proof}
If $\BigCell \horizontal \Cell$ then $\BigCell\Min = \Cell\Min = \Cell$ by Observation~\ref{obs:min_min} so $\Cell \Up \BigCell$. Otherwise $\BigCell \Down \Cell$.
\end{proof}

The next two lemmas show transitivity of $\Up$ and $\Down$ so that we can restrict our attention to alternating sequences of moves.

\begin{lem}
\label{lem:up_transitive}
It never happens that $\Cell_1 \Up \Cell_2 \Up \Cell_3$. In particular $\Up$ is transitive.
\end{lem}

\begin{proof}
Suppose $\Cell_1 \Up \Cell_2 \Up \Cell_3$. Then by Observation~\ref{obs:min_min} $\Cell_1 = \Cell_2\Min = (\Cell_3\Min)\Min = \Cell_3\Min = \Cell_2$ contradicting $\Cell_1 \ne \Cell_2$.
\end{proof}

\begin{lem}
\label{lem:down_transitive}
The relation $\Down$ is transitive.
\end{lem}

\begin{proof}
Assume $\Cell_1 \Down \Cell_2 \Down \Cell_3$. Clearly $\Cell_1 \gneq \Cell_3$. Suppose $\Cell_1 \horizontal \Cell_3$, then by Proposition~\ref{prop:horizontal_properties}~\eqref{item:faces} $\Cell_2 \horizontal \Cell_3$ contradicting $\Cell_2 \Down \Cell_3$.
\end{proof}

Now we approach the proof that alternating sequences of moves have bounded length.

\begin{lem}
If
\[
\Cell_1 \Up \BigCell_1 \Down \Cell_2
\]
then
\[
\Cell_1 = (\Cell_1 \vee \Cell_2)\Min \quad \text{and} \quad \Cell_1 \vee \Cell_2 \Down \Cell_2 \text{ .}
\]
In particular $\Cell_1 \Up \Cell_1 \vee \Cell_2 \Down \Cell_2$ unless $\Cell_1 \Down \Cell_2$.
\end{lem}

\begin{proof}
By Proposition~\ref{prop:horizontal_properties}~\eqref{item:faces} $\Cell_1 \vee \Cell_2 \horizontal \Cell_1$ so $\Cell_1 = (\Cell_1 \vee \Cell_2)\Min$ by Observation~\ref{obs:min_min}. And $\Cell_1 = \BigCell_1\Min \not\le \Cell_2$ whence $\Cell_2 \lneq \Cell_1 \vee \Cell_2$ so that $\Cell_1 \vee \Cell_2 \Down \Cell_2$.
\end{proof}

\begin{lem}
\label{lem:min_in_vertical_link}
If $\BigCell \ge \Cell$ are horizontal cells then $(\BigCell\Min \vee \Cell) \direction \Cell \subseteq \Vertical\Link \Cell$.
\end{lem}

\begin{proof}
Suppose not and let $\Another\Cell$ be the maximal face of $\BigCell\Min \vee \Cell$ such that $\Another\Cell \direction \Cell \subseteq \Vertical\Link \Cell$. Let $\YetAnother\Cell$ be the face of $\BigCell\Min \vee \Cell$ such that $\YetAnother\Cell \direction \Cell = ((\BigCell\Min \vee \Cell) \direction \Cell) \setminus (\Another\Cell \direction \Cell)$. Then $\YetAnother\Cell \horizontal \Cell$ so by Proposition~\ref{prop:horizontal_properties}~\eqref{item:join} $\BigCell\Min \vee \Cell = \YetAnother\Cell \vee \Another\Cell \horizontal \Cell \vee \Another\Cell = \Another\Cell$. But $\BigCell \horizontal \BigCell\Min \vee \Cell$ by so Proposition~\ref{prop:horizontal_properties}~\eqref{item:join}, so Proposition~\ref{prop:horizontal_properties}~\eqref{item:transitivity} shows that $\BigCell \horizontal \Another\Cell$.
Finally Proposition~\ref{prop:horizontal_properties}~\eqref{item:non_empty_meet} implies $\BigCell \horizontal \Another\Cell \intersect \BigCell\Min$ which contradicts minimality of $\BigCell\Min$ because $\Another\Cell \intersect \BigCell\Min$ is a proper face.
\end{proof}

\begin{cor}
If
\[
\Cell_1 \Up \BigCell_1 \Down \Cell_2 \Up \BigCell_2
\]
then $\BigCell_2 \vee \Cell_1$ exits.
\end{cor}

\begin{proof}
By assumption we have $\BigCell_2 \direction \Cell_2 \subseteq \Horizontal\Link \Cell_2$. Since $\Cell_1 = \BigCell_1\Min$, Lemma~\ref{lem:min_in_vertical_link} shows that $(\Cell_1 \vee \Cell_2) \direction \Cell_2 \subseteq \Vertical\Link \Cell_2$. Hence $(\BigCell_2 \direction \Cell_2) \vee ((\Cell_1 \vee \Cell_2) \direction \Cell_2)$ exists and so does $\Cell_1 \vee \BigCell_2$.
\end{proof}

\begin{lem}
\label{lem:cell1_bigcell2_min}
If
\[
\Cell_1 \Up \BigCell_1 \Down \Cell_2 \Up \BigCell_2
\]
then $(\BigCell_2 \vee \Cell_1)\Min = \Cell_1$.
\end{lem}

\begin{proof}
By assumption $\BigCell_2 \horizontal \Cell_2$ so Proposition~\ref{prop:horizontal_properties}~\eqref{item:join} implies $\BigCell_2 \vee \Cell_1 \horizontal \Cell_2 \vee \Cell_1$. Moreover $\BigCell_1 \horizontal \Cell_2 \vee \Cell_1$ so $(\Cell_2 \vee \Cell_1)\Min = \BigCell\Min = \Cell_1$ by Observation~\ref{obs:min_min}. Thus $(\BigCell_2 \vee \Cell_1)\Min = \Cell_1$ again by Observation~\ref{obs:min_min}.
\end{proof}

\begin{lem}
\label{lem:shortening}
Any alternating chain
\[
\Cell_1 \Up \BigCell_1 \Down \Cell_2 \Up \BigCell_2
\]
can be shortened to either
\[
\Cell_1 \Up \Cell_1 \vee \BigCell_2 \Down \BigCell_2 \quad \text{or} \quad \Cell_1 \Up \BigCell_1 \Down \BigCell_2 \text{ .}
\]
\end{lem}

\begin{proof}
We know by assumption that $\BigCell_2\Min = \Cell_2$ and Lemma~\ref{lem:cell1_bigcell2_min} implies $(\Cell_1 \vee \BigCell_2)\Min = \Cell_1$. Since by Observation~\ref{obs:no_up_down_cycle} $\Cell_1 \ne \Cell_2$ this implies $\BigCell_2 \ne \Cell_1 \vee \BigCell_2$ so that $\Cell_1 \vee \BigCell_2 \Down \BigCell_2$.

If $\Cell_1 \ne \Cell_1 \vee \BigCell_2$, then $\Cell_1 \Up \Cell_1 \vee \BigCell_2$. If $\Cell_1 = \Cell_1 \vee \BigCell_2$, then $\BigCell_1 \gneq \Cell_1 \gneq \BigCell_2$ so that $\BigCell_1 \Down \BigCell_2$.
\end{proof}

\begin{cor}
\label{cor:no_cycles}
No sequence of moves enters a cycle.
\end{cor}

\begin{proof}
Since $\Up$ and $\Down$ are transitive by Lemma~\ref{lem:up_transitive} respectively \ref{lem:down_transitive}, a cycle of minimal length must be alternating. Thus by Lemma~\ref{lem:shortening} it can go up at most once. But then it would have to be of the form ruled out by Observation~\ref{obs:no_up_down_cycle}.
\end{proof}

\begin{lem}
\label{lem:vee_of_lower_terms_exists}
If
\[
\Cell_1 \Up \BigCell_1 \Down \cdots \Down \Cell_{k-1} \Up \BigCell_{k-1} \Down \Cell_k
\]
then $\Cell_1 \vee \cdots \vee \Cell_k$ exists.
\end{lem}

\begin{proof}
First we show by induction that $\Cell_1 \vee \Cell_k$ exists. If $k=2$ this is obvious. Longer chains can be shortened using Lemma~\ref{lem:shortening}.

Applying this argument to subsequences we see that $\Cell_i \vee \Cell_j$ exists for any two indices $i$ and $j$. So by Observation~\ref{obs:polyflag} $\Cell_1 \vee \cdots \vee \Cell_k$ exists.
\end{proof}

\begin{proof}[Proof of Proposition \ref{prop:bound_on_moves}]
We first consider the case of alternating sequences. By Lemma~\ref{lem:vee_of_lower_terms_exists} for any alternating sequence of moves, the join of the lower elements (to which there is a move going down or from which there is a move going up) exists. This cell has at most $N \defeq \prod_i (2^{\dim(\EuclBuilding_i)+1}-1)$ non-empty faces. Since by Corollary~\ref{cor:no_cycles} the alternating sequence cannot contain any cycles $N$ is also the maximal number of moves.

A sequence of moves going down can at most have length $\dim(\Space) = \sum_{i} \dim(\EuclBuilding_i)$ while there are no two consecutive moves up by Lemma~\ref{lem:up_transitive}.

So an arbitrary sequence of moves can contain go up at most $N$ times and can go down at most $\dim(\Space) \cdot N$ times making $(\dim(\Space)+1) \cdot N$ a bound on its length (counting moves, not cells).
\end{proof}

\subsection*{Proof of Proposition \ref{prop:horizontal_properties} using spherical geometry}

In this part we prove Proposition~\ref{prop:horizontal_properties} using spherical geometry. This proof has the advantage of being elementary but on the other hand it is quite technical. We start by collecting some properties of spherical triangles (we only consider unit-spheres, i.e.\ spheres of curvature $1$):

\begin{obs}
\label{obs:spherical_triangles}
Consider a spherical triangle (with edges of length $< \pi$) with vertices $a$, $b$, and $c$ and angles $\alpha = \angle_a(b,c)$, $\beta = \angle_b(a,c)$ and $\gamma = \angle_c(a,b)$.
\begin{enumerate}
\item If $d(a,b) = \pi/2$ and $d(b,c),d(a,c) \le \pi/2$, then $\gamma \ge \pi/2$. \label{item:edge_gives_angle}
\item If $d(a,b) = \pi/2$ and $\beta = \pi/2$, then $d(a,c) = \gamma = \pi/2$.\\
If $d(a,b) = \pi/2$ and $\beta < \pi/2$, then $d(a,c) < \pi/2$. \label{item:edge_and_angle_give_edge_and_angle}
\item If $d(a,b) = d(a,c)= \pi/2$ and $b \ne c$, then $\beta = \gamma = \pi/2$. \label{item:edges_give_angles}
\item If $\beta = \gamma = \pi/2$ and $b \ne c$, then $d(a,b) = d(a,c) = \pi/2$. \label{item:angles_give_edges}
\end{enumerate}
\end{obs}

\begin{proof}
All properties can be deduced from the spherical law of cosines, see \cite[I.2.2]{brihae}. But they can also easily be verified geometrically. We illustrate this for the third statement. Put $b$ and $c$ on the equator of a $2$-sphere. The two great circles that meet the equator perpendicularly in $b$ and $c$ only meet at the poles, which have distance $\pi/2$ from the equator.
\end{proof}

\begin{figure}[!ht]
\begin{center}
\includegraphics{sphere_link_identification}
\end{center}
\caption{How $\Link\BigCell$ is identified with $\Link \BigCell \direction \Cell$.}
\label{fig:sphere_link_identification}
\end{figure}

A recurrent motif in the proof we are approaching will be the following: Let $\SpherBuilding$ be a spherical building, let $\Cell \subseteq \SpherBuilding$ be a cell and let $\BigCell$ be a coface of $\Cell$. Then the link of $\BigCell$ in $\SpherBuilding$ can be identified with the link of $\BigCell \direction \Cell$ in $\Link \Cell$. For example assume that $\Cell$ is a vertex and that $\BigCell$ is an edge, and let $\Point$ be a vertex adjacent to $\BigCell$. Then the direction that $\Point$ defines in the link of $\BigCell$ is identified with the direction from $\BigCell \direction \Cell$ toward $[\Cell,x] \direction \Cell$ (see Figure~\ref{fig:sphere_link_identification}).

Recall from Lemma~\ref{lem:spherical_projection} that if $\Cell$ is convex and $\Point$ has distance $< \pi/2$ to $\Cell$, then the projection point $\SpherPoint_\Point$ of $\Point$ to $\Cell$ is the unique point in $\Cell$ that is closest to $\Point$ and that the direction $[\SpherPoint_\Point,\Point]_\Point$ is perpendicular to $\Cell$, so that it defines a direction of its link.

A problem we will be facing is how to compute the distance of two points in the link of a cell that is not a vertex. For example assume that $\Cell$ is an edge and let $\Point$ and $\AltPoint$ be vertices adjacent to $\Cell$. Then we want to determine the distance of $(\Point \vee \Cell) \direction \Cell$ and $(\AltPoint \vee \Cell) \direction \Cell$ in the link of $\Cell$. This is of course the angle at $\Cell$ between the triangles $\Point \vee \Cell$ and $\AltPoint \vee \Cell$. Let us denote it by $\angle_\Cell(\Point,\AltPoint)$.

If $\Point$ and $\AltPoint$ project to the same point in $\Cell$, say $\SpherPoint$, then there is no problem: the desired value is just $\angle_\SpherPoint(\Point,\AltPoint)$. If either of the points $\Point$ and $\AltPoint$ has distance $\pi/2$ to $\Cell$ (so that it has no projection point), then there is still no problem: the direction from any point of $\Cell$ is perpendicular to $\Cell$ by Observation~\ref{obs:spherical_triangles}~\eqref{item:edges_give_angles}. In general however, $\Point$ and $\AltPoint$ will both have unique projection points $\SpherPoint_\Point$ and $\SpherPoint_\AltPoint$ in $\Cell$ and they will be distinct. Luckily we can avoid to deal with these cases, see Remark~\ref{remnum:induction_vs_spherical_geometry} below.

Let $\SpherBuilding$ be a spherical building with chosen north pole. We will repeatedly use the criterion from Lemma~\ref{lem:geometric_criterion}:
$\Cell \subseteq \Horizontal\SpherBuilding$ if and only if $\SpherDistance(\AltVertex,\Vertex) = \pi/2$ for every vertex $\AltVertex$ of $\Cell$ and every non-equatorial vertex $\Vertex$ adjacent to $\Cell$. Note that there is always an apartment that contains $\Cell$ and $\Vertex$ so the branching of the building is not an issue. Note also that since apartments are Coxeter complexes, all cells have diameter at most $\pi/2$. If $\Cell$ is a cell and $\Vertex$ and $\AltVertex$ are vertices adjacent to $\Cell$ and to each other, then $\angle_\Cell(\Vertex,\AltVertex) \le \pi/2$: this is just the diameter of an edge in $\Link \Cell$.

Before we start with the actual proof, we need to record one more fact:

\begin{obs}
Let $\Busemann$ be a Busemann function on $\Space$ and let $\Cell \le \BigCell$ be cells that are horizontal with respect to $\Busemann$. Let $\NorthPole_\Cell$ and $\NorthPole_\BigCell$ be the north poles determined by $\Busemann$ in $\Link \Cell$ and $\Link \BigCell$ respectively. The identification of $\Link \BigCell$ with $\Link (\BigCell \direction \Cell)$ identifies $\NorthPole_\BigCell$ with the direction toward $\NorthPole_\Cell$.

Moreover a vertex $\Vertex$ adjacent to $\BigCell \direction \Cell$ in $\Link \Cell$ is equatorial in $\Link \Cell$ if and only if it is equatorial as a vertex of $\Link (\BigCell \direction \Cell)$.
\end{obs}

\begin{proof}
The first statement is clear. For the second note that $\BigCell \direction \Cell$ is equatorial in $\Link \Cell$, i.e.\ that $\SpherDistance(\NorthPole_\Cell,\BigCell \direction \Cell) = \pi/2$. Let $\SpherPoint$ be a point in $\BigCell \direction \Cell$ that is closest to $\Vertex$ (if $\SpherDistance(\Vertex,\BigCell) < \pi/2$ it is the unique projection point, otherwise any point). Then $\angle_{\BigCell \direction \Cell}(\Vertex,\NorthPole_\Cell) = \angle_\SpherPoint(\Vertex,\NorthPole_\Cell)$. Now Observation~\ref{obs:spherical_triangles}~\eqref{item:edge_and_angle_give_edge_and_angle} and \eqref{item:edges_give_angles} imply that $\Vertex$ is equatorial if and only if $\angle_\SpherPoint(\Vertex,\NorthPole_\Cell) = \pi/2$.
\end{proof}

\begin{rem}
It is helpful for the understanding of the proof above to first think about the case where $\Cell$ is a facet (face of codimension $1$) of $\BigCell$, so that $\BigCell \direction \Cell$ is a vertex. Similarly in the proof of Proposition~\ref{prop:horizontal_properties}~\eqref{item:transitivity} below, it is helpful to first think of the case where $\Cell$ is a facet of $\Another\Cell$.
\end{rem}

Now we start with the proof of Proposition~\ref{prop:horizontal_properties}. The first statement is trivial.

\begin{proof}[Proof of Proposition~\ref{prop:horizontal_properties}~\eqref{item:join}]
Let $\SpherBuilding$ be the link of $\Cell$ with a fixed north pole. Assume first that $\Cell$ is a facet of $\Another\Cell \vee \Cell$. Let $\overline{\Another\Cell} = (\Cell \vee \Another\Cell) \direction \Cell$ (note that this is a vertex by the assumption we just made), let $\Vertex$ be a non-equatorial vertex adjacent to $\BigCell \direction \Cell$ and let $\AltVertex$ be a vertex of $\BigCell \direction \Cell$ other than $\overline{\Another\Cell}$. Note that $\AltVertex$ corresponds to a vertex of $(\BigCell \vee \Another\Cell) \direction (\Cell \vee \Another\Cell)$ and each vertex of $(\BigCell \vee \Another\Cell) \direction (\Cell \vee \Another\Cell)$ is obtained in this way.

Using the criterion Lemma~\ref{lem:geometric_criterion} and the identification of $\Link \Cell \vee \Another\Cell$ with $\Link \overline{\Another\Cell}$ we know that $\SpherDistance(\Vertex,\AltVertex) = \pi/2$ and we have to show that  $\angle_{\overline{\Another\Cell}}(\Vertex,\AltVertex) = \pi/2$.

The triangle with vertices $\Vertex$, $\AltVertex$, and $\overline{\Another\Cell}$ satisfies $\SpherDistance(\Vertex,\AltVertex) = \pi/2$ and the angle at $\overline{\Another\Cell}$ can be at most $\pi/2$ because we are considering cells in a Coxeter complex. Hence by Observation~\ref{obs:spherical_triangles}~\eqref{item:edge_gives_angle} it has to be precisely $\pi/2$ as desired.

For the general case set $\Another\Cell_0 \defeq \Another\Cell \vee \Cell$ and inductively take $\Another\Cell_{i+1}$ to be a facet of $\Another\Cell_i$ that contains $\Cell$ until $\Another\Cell_n = \Cell$ for some $n$.

By assumption $\BigCell = \BigCell \vee \Another\Cell_n \horizontal \Cell \vee \Another\Cell_n = \Cell$ and the above argument applied to $\Another\Cell_{n-1}$ shows that $\BigCell \vee \Another\Cell_{n-1} \horizontal \Cell \vee \Another\Cell_{n-1}$. Proceeding inductively we eventually obtain $\BigCell \vee \Another\Cell_0 \horizontal \BigCell \vee \Another\Cell_0$ which is what we want.
\end{proof}

\begin{proof}[Proof of Proposition~\ref{prop:horizontal_properties}~\eqref{item:transitivity}]
Let $\SpherBuilding$ be the link of $\Cell$, let $\overline{\Another\Cell} = \Another\Cell \direction \Cell$, let $\Vertex$ be a non-equatorial vertex in $\SpherBuilding$ adjacent to $\overline{\Another\Cell}$ and let $\AltVertex$ be any vertex of $\BigCell \direction \Cell$ not contained in $\overline{\Another\Cell}$.

According to our criterion Lemma~\ref{lem:geometric_criterion} we know that $\SpherDistance(\Vertex,\overline{\Another\Cell}) = \pi/2$ (because $\Another\Cell \horizontal \Cell$) and that $\angle_{\overline{\Another\Cell}}(\Vertex,\AltVertex) = \pi/2$ (because $\BigCell \horizontal \Another\Cell$) and we have to show that $\SpherDistance(\Vertex,\AltVertex) = \pi/2$.

Let $\SpherPoint$ be the projection of $\AltVertex$ to $\overline{\Another\Cell}$ if it exists or otherwise take any point of $\overline{\Another\Cell}$. Then $[\SpherPoint,\AltVertex]$ is perpendicular to $\overline{\Another\Cell}$ by the choice of $\SpherPoint$ and $[\SpherPoint,\Vertex]$ is perpendicular to $\overline{\Another\Cell}$ because $\SpherDistance(\Vertex,\overline{\Another\Cell}) = \pi/2$. Thus $\angle_{\overline{\Another\Cell}}(\Vertex,\AltVertex) = \angle_\SpherPoint(\Vertex,\AltVertex)$.

We consider the triangle with vertices $\Vertex$, $\AltVertex$, and $\SpherPoint$. We know that $\SpherDistance(\Vertex,\SpherPoint) = \pi/2$ and that $\angle_\SpherPoint(\Vertex,\AltVertex) = \pi/2$. From this we deduce that $\SpherDistance(\Vertex,\AltVertex) = \pi/2$ using Observation~\ref{obs:spherical_triangles}~\eqref{item:edge_and_angle_give_edge_and_angle}.
\end{proof}

\begin{proof}[Proof of Proposition~\ref{prop:horizontal_properties}~\eqref{item:non_empty_meet}]
Because we have already proven transitivity of $\horizontal$, it suffices to show that $\Cell_1 \horizontal \Cell_1 \intersect \Cell_2$. We assume first that $\Cell_1 \intersect \Cell_2$ is a facet of both, $\Cell_1$ and $\Cell_2$.

Let $\SpherBuilding$ be the link of $\Cell_1 \intersect \Cell_2$ and let $\overline{\Cell_i} = \Cell_i \direction (\Cell_1 \intersect \Cell_2)$. Note that $\overline{\Cell_1}$ and $\overline{\Cell_2}$ are distinct vertices. Let $\Vertex$ be a non-equatorial vertex in $\SpherBuilding$ adjacent to $\overline{\Cell_1}$.

By the criterion in Lemma~\ref{lem:geometric_criterion} we know that $\angle_{\overline{\Cell_2}}(\Vertex,\overline{\Cell_1}) = \pi/2 = \angle_{\overline{\Cell_1}}(\Vertex,\overline{\Cell_2})$ and we have to show that $\SpherDistance(\Vertex,\overline{\Cell_1}) = \pi/2$. But this follows at once from Observation~\ref{obs:spherical_triangles}~\eqref{item:angles_give_edges}.

%For the general case set $\Cell_2^0 \defeq \Cell_1 \vee \Cell_2$ and inductively let $\Cell_2^{i+1}$ be a facet of $\Cell_2^i$ that contains $\Cell_2$ until $\Cell_2^n = \Cell_2$ for some $n$. Also let $\Cell_1^{i+1} \defeq \Cell_1 \intersect \Cell_2^i$. Then $\Cell_1^{i+1} \intersect \Cell_2^{i+1}$ is a facet of both, $\Cell_1^i$ and $\Cell_2^i$ and $\Cell_1^n \intersect \Cell_2^n = \Cell_1 \intersect \Cell_2$.
%
%By Proposition~\ref{prop:horizontal_properties}~\eqref{item:faces} $\BigCell \horizontal \Cell_2^i$ for all $i$. Our argument above applied to $\Cell_1^0$ and $\Cell_2^0$ shows that $\BigCell \horizontal \Cell_1^1$. Proceeding inductively we eventually obtain $\BigCell \horizontal \Cell_1^n$ as desired.
Now we consider the more general case where $\Cell_1 \intersect \Cell_2$ is a facet of $\Cell_2$ but need not be a facet of $\Cell_1$. Set $\Cell_2^0 \defeq \Cell_1 \vee \Cell_2$ and let $\Cell_2^{i+1}$ be a facet of $\Cell_2^i$ that contains $\Cell_2$ until $\Cell_2^n = \Cell_2$ for some $n$. Also let $\Cell_1^{i+1} = \Cell_1^i \intersect \Cell_2^i$ for $0 \le i \le n-1$. Then $\Cell_1^i \intersect \Cell_2^i$ is a facet of both, $\Cell_1^i$ and $\Cell_2^i$ for $1 \le i \le n$. By assumption $\BigCell \horizontal \Cell_1$ and Proposition~\ref{prop:horizontal_properties}~\eqref{item:join} shows that $\BigCell \horizontal \Cell_2^i$ for $0 \le i \le n$. Thus the argument above implied inductively shows that $\Cell_1 \horizontal \Cell_1 \intersect \Cell_2$.

An analogous induction allows to drop the assumption that $\Cell_1 \intersect \Cell_2$ be a facet of $\Cell_2$.
\end{proof}

\begin{remnum}
\label{remnum:induction_vs_spherical_geometry}
We claimed at the beginning of Section~\ref{sec:horizontal_links} that the entire argument would be carried out in the links. Now the reader has found in the proofs of Proposition~\ref{prop:horizontal_properties}~\eqref{item:join} and \eqref{item:non_empty_meet} inductions that involve cells of the Euclidean building. It is possible to avoid this by investing more spherical geometry. At this point however we favored simplicity over purity.

The alternative proof below is truly carried out in the links.
\end{remnum}

\subsection*{Proof of Proposition \ref{prop:horizontal_properties} using Coxeter diagrams}

In this part we prove Proposition \ref{prop:horizontal_properties} using Coxeter diagrams. We use the criterion Lemma~\ref{lem:diagram_criterion}: Let $\SpherBuilding$ be a spherical building with chosen north pole, let $\Cell \subseteq \SpherBuilding$ be a simplex and let $\Chamber$ be a chamber that contains $\Cell$. Then $\Cell \subseteq \Horizontal\SpherBuilding$ if and only if the following holds: for every vertex $\AltVertex$ of $\Cell$ and every vertex $\Vertex$ of $\Chamber$, if there is a path in $\typ\SpherBuilding$ from $\typ\Vertex$ to $\typ \AltVertex$, then $\Vertex$ is equatorial.

Let $\Busemann$ be a Busemann function on $\Space$, let $\Another\Cell \ge \Cell$ be horizontal cells (with respect to $\Busemann$). Then $\Busemann$ defines a north pole in $\Link \Another\Cell$ as well as in $\Link\Cell$. If $\BigCell$ is a coface of $\Another\Cell$, then $\BigCell \direction \Cell$ is equatorial if and only if $\BigCell$ is horizontal if and only if $\BigCell \direction \Another\Cell$ is equatorial.

In the above situation $\typ \Link \Another\Cell$ can be considered as the sub-diagram obtained from $\typ \Link \Cell$ by removing $\typ \Another\Cell$. What we have just seen is:
\begin{obs}
Let $\Chamber$ be a chamber that contains horizontal cells $\Another\Cell \ge \Cell$. If $\typ \Vertex = \typ \Another\Vertex$ for vertices $\Vertex$ of $\Chamber \direction \Cell$ and $\Another\Vertex$ of $\Chamber \direction \Another\Cell$, then $\Vertex$ is equatorial if and only if $\Another\Vertex$ is.\qed
\end{obs}
So once we have chosen a chamber $\Chamber$ and a non-empty cell $\Cell$, we may think of being horizontal as a property of the nodes of $\typ\Link\Cell$.

We come to the proof of Proposition \ref{prop:horizontal_properties}. The first statement is again trivial.

\begin{proof}[Proof of Proposition~\ref{prop:horizontal_properties}~\eqref{item:join}]
Let $\Chamber$ be a chamber that contains $\BigCell \vee \Another\Cell$. Let $\Vertex$ be a non-equatorial vertex of $\Chamber \direction \Cell \vee \Another\Cell$ and let $\AltVertex$ be a vertex of $(\BigCell \vee \Another\Cell) \direction (\Cell \vee \Another\Cell)$. Note that $\typ (\Link \Cell \vee \Another\Cell)$ is obtained from $\typ \Link\Cell$ by removing $\typ \Another\Cell$.

Assume that there were a path in $\typ \Link \Cell \vee \Another\Cell$ that connects $\typ\Vertex$ to $\typ\AltVertex$. Then this would in particular be a path in $\typ \Link \Cell$ contradicting $\BigCell \horizontal \Cell$ by Lemma~\ref{lem:diagram_criterion}.
\end{proof}

\begin{proof}[Proof of Proposition~\ref{prop:horizontal_properties}~\eqref{item:transitivity}]
Let $\Chamber$ be a chamber that contains $\BigCell$, let $\Vertex$ be a non-equatorial vertex of $\Chamber \direction \Cell$ and let $\AltVertex$ be a vertex of $\BigCell \direction \Cell$. Note that $\typ \Link \Another\Cell$ is obtained from $\typ \Link\Cell$ by removing $\typ \Another\Cell$.

Assume that there were a path in $\typ \Link \Cell$ from $\typ \Vertex$ to $\typ \AltVertex$. Then either this path does not meet $\typ \Another\Cell$, thus lying entirely in $\typ \Link \Another\Cell$ and therefore contradicting $\BigCell \horizontal \Another\Cell$ by Lemma~\ref{lem:diagram_criterion}. Or there would be a first vertex in the path that lies in $\typ \Another\Cell$, say $\typ \Another\AltVertex$.  Then the subpath from $\typ\Vertex$ to $\typ\Another\AltVertex$ would lie in $\typ\Link\Cell$ and thus contradict $\Another\Cell \horizontal \Cell$ by Lemma~\ref{lem:diagram_criterion}.
\end{proof}

\begin{proof}[Proof of Proposition~\ref{prop:horizontal_properties}~\eqref{item:non_empty_meet}]
Let $\Chamber$ be a chamber that contains $\BigCell$, let $\Vertex$ be a non-equatorial vertex of $\Chamber \direction (\Cell_1 \intersect \Cell_2)$ and let $\AltVertex$ be a vertex of $\BigCell \direction (\Cell_1 \intersect \Cell_2)$. Again $\typ \Link \Cell_i$ is obtained from $\typ \Link (\Cell_1 \intersect \Cell_2)$ by removing $\typ \Cell_i$. Note also that $\typ (\Cell_1 \direction (\Cell_1 \intersect \Cell_2))$ and $\typ (\Cell_2 \direction (\Cell_1 \intersect \Cell_2))$ are disjoint.

Assume that there were a path in $\typ \Link (\Cell_1 \intersect \Cell_2)$ from $\typ \Vertex$ to $\typ \AltVertex$. Let $\typ \Another\AltVertex$ be the first vertex in $\typ (\Cell_1 \direction (\Cell_1 \intersect \Cell_2))\union \typ (\Cell_2 \direction (\Cell_1 \intersect \Cell_2)))$ that the path meets and assume without loss of generality that $\typ \Another\AltVertex \in \typ (\Cell_1 \direction (\Cell_1 \intersect \Cell_2))$. Then the subpath from $\typ \Vertex$ to $\typ \Another\AltVertex$ lies entirely in $\typ \Link \Cell_2$ contradicting $\BigCell \horizontal \Cell_2$ by Lemma~\ref{lem:diagram_criterion}.
\end{proof}
% ***************
% *** SECTION ***
% ***************

\section{Zonotopes}
\label{sec:zonotopes}

In this section we reproduce and extend the contents of \cite[Section 2]{buxgrawit10}. Let $\EuclSpace$ be a Euclidean vector space with scalar product $\scaprod{\cdot}{\cdot}$ and metric $\EuclDistance$. The relative interior $\relint \Face$ of a polyhedron $\Face$ in $\EuclSpace$ is the interior of $\Face$ in its affine span. It is obtained from $\Face$ by removing all proper faces.

Let $\Zonotope \subseteq \EuclSpace$ be a convex polytope. We denote by $\Proj[\Zonotope]$ the closest point-projection, i.e.\ $\Proj[\Zonotope] \Point = \AltPoint$ if $y$ is the point in $\Zonotope$ closest to $\Point$. The \emph{normal cone} of a non-empty face $\Face$ of $\Zonotope$ is the set
\[
\NormalCone\Face \defeq \{\Vector \in \EuclSpace \mid \scaprod{\Vector}{\Point} = \max_{\AltPoint \in \Zonotope} \scaprod{\Vector}{\AltPoint} \text{ for every }\Point \in \Face \} \text{ .}
\]
The significance of this notion for us is:
\begin{obs}
\label{obs:space_decomposition}
The space $\EuclSpace$ decomposes as a disjoint union
\[
\EuclSpace = \Union_{\emptyset \ne \Face \le \Zonotope} \relint \Face + \NormalCone\Face
\]
with $(\Face - \Face) \intersect(\NormalCone\Face - \NormalCone\Face) = \{0\}$
and if $\Point$ is written in the unique way as $\FaceVector + \NormalVector$ according to this decomposition, then $\Proj[\Zonotope] \Point = \FaceVector$.\qed
\end{obs}

\begin{figure}[!ht]
\begin{center}
\includegraphics{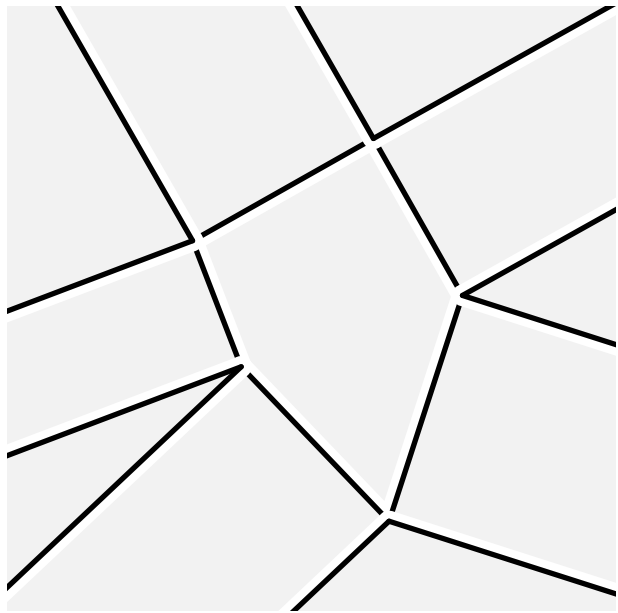}
\end{center}
\caption{The decomposition given by Observation~\ref{obs:space_decomposition}: The shaded regions are the classes of the partition. The white strips are to emphasize to which class the boundary points (black) belong.}
\label{fig:projection_decomposition}
\end{figure}

We are interested in the situation where $\Zonotope$ is a zonotope.
% (see for example \cite{mcmullen71}).
For our purpose a \emph{zonotope} is described by a finite set $\Directions \subseteq \EuclSpace$ and defined to be
\[
\Zonotope[\Directions] = \sum_{\Direction \in \Directions} [0,\Direction]
\]
where the sum is the Minkowski sum ($\ConvexSet_1 + \ConvexSet_2 = \{\Vector_1 + \Vector_2 \mid \Vector_1 \in \ConvexSet_1, \Vector_2 \in \ConvexSet_2\}$). The faces of zonotopes are themselves translates of zonotopes, they have the following nice description (\cite[Lemma~2.3]{buxgrawit10}):

\begin{lem}
If $\Face$ is a face of $\Zonotope[\Directions]$ and $\Vector \in \relint \NormalCone\Face$, then
\[
\Face = \Zonotope[\Directions_\Vector] + \sum_{\substack{\Direction \in \Directions \\ \scaprod{\Vector}{\Direction} > 0}} \Direction \text{ .}
\]
where $\Directions_\Vector \defeq \{\Direction \in \Directions \mid \scaprod{\Vector}{\Direction} = 0\}$.
\end{lem}

It is a basic fact from linear optimization that the relative interiors of normal cones of non-empty faces of $\Zonotope$ partition $\EuclSpace$, so for every non-empty face $\Face$ such a vector $\Vector$ exists and vice versa.

The zonotopes we will be concerned with have the following interesting property:

\begin{prop}
\label{prop:zonotope_contains_parallel_translate}
Let let $\Cell$ be a polytope and let $\Directions$ be a finite set of vectors that has the property that $\AltVertex - \Vertex \in \Directions$ for any two vertices $\Vertex, \AltVertex$ of $\Cell$. For a vertex $\Vertex$ of $\Cell$ let $E_\Vertex$ be the set of vectors $\AltVertex - \Vertex$ for vertices $\AltVertex \ne \Vertex$ of $\Cell$. Then for every point $\Point$ of $\Zonotope[\Directions]$ there is a vertex $\Vertex$ of $\Cell$ such that $\Point + \Zonotope[E_v] \subseteq \Zonotope[\Directions]$. In particular $\Zonotope[\Directions]$ contains a parallel translate of $\Cell$ that contains $\Point$.
\end{prop}

\begin{proof}
It is not hard to see that we may assume $\Directions$ to contain precisely the vectors $\AltVertex - \Vertex$ with $\Vertex, \AltVertex$ vertices of $\Cell$ and we do so. Write
\begin{equation}
\label{eq:point_sum}
\Point = \sum_{z \in D} \alpha_z z
\end{equation}
with $0 \le \alpha_z \le 1$. We consider the complete directed graph whose vertices are the vertices of $\Cell$ and label the edge from $v$ to $w$ by $\alpha_{w-v}$.

If there is a cycle in this graph, all edges of which have a strictly positive label, then we may subtract the minimum of the labels from all $\alpha_{w-v}$ for which the cycle contains the edge from $v$ to $w$ and \eqref{eq:point_sum} remains true. Moreover at least one edge in the cycle is now labelled by $0$. Note that this is possible even if two edges are parallel, i.e.\ $w - v = w' - v'$. Iterating this procedure we eventually obtain a graph that does not contain any cycles with strictly positive labels. In particular there is a vertex whose outgoing edges are all labelled by $0$ because there are only finitely many vertices.

Let $v$ be such a vertex. Then $\alpha_z = 0$ for $z \in E_v$. Thus $\Point = \sum_{z \in D \setminus E_v} \alpha_{z} z$ and $\Point + \Zonotope(E_v) \subseteq \Zonotope(D)$.

For the last statement note that $x \in (x-v) + \Cell \subseteq x + \Zonotope(E_v)$ because $x + \Zonotope(E_v)$ is convex and contains all vertices of $(x-v) + \Cell$.
\end{proof}

We say that a finite set of vectors $\Directions$ is \emph{sufficiently rich} for a polytope $\Cell$ if it satisfies the hypothesis of Proposition~\ref{prop:zonotope_contains_parallel_translate}, i.e.\ for any two distinct vertices $\Vertex$ and $\AltVertex$ of $\Cell$ the vector $\AltVertex - \Vertex$ is in $\Directions$. Trivially if $\Directions$ is sufficiently rich for $\Cell$ then it is sufficiently rich for the convex hull of any set of vertices of $\Cell$. Note that the property in the conclusion of Proposition~\ref{prop:zonotope_contains_parallel_translate} is not hereditary in that way: a square contains for each of its points a parallel translate of itself that contains that point; but it does not contain for every point a parallel translate of a diagonal that contains that point.

\begin{prop}
\label{prop:sufficiently_rich_min_vertex}
If $\Directions$ is sufficiently rich for a polytope $\Cell$ then among the points of $\Cell$ closest to $\Zonotope[\Directions]$ there is a vertex. Moreover the points farthest from $\Zonotope[\Directions]$ form a face of $\Cell$.
\end{prop}

\begin{proof}
Let $\Point \in \Cell$ be a point that minimizes distance to $\Zonotope[\Directions]$. Proceeding inductively it suffices to find a point in a proper face of $\Cell$ that has the same distance. Let $\bar{\Point} = \Proj[{\Zonotope[\Directions]}] \Point$. By Proposition~\ref{prop:zonotope_contains_parallel_translate} $\Zonotope[\Directions]$ contains a translate $\Another\Cell$ of $\Cell$. All points in $\Cell \intersect (\Point - \bar{\Point} + \Another\Cell)$ have the same distance to $\Zonotope[\Directions]$ as $\Point$. And since this set is the non-empty intersection of $\Cell$ with a translate, it contains a boundary point of $\Cell$.

For the second statement note that if $\EuclDistance(\Zonotope[\Directions],\cdot)$ attains its maximum over $\conv V$ in a relatively interior point then it is in fact constant on $\conv V$ by convexity. Now if $V$ is a set of vertices of $\Cell$ on which $\EuclDistance(\Zonotope[\Directions],\cdot)$ is maximal, we apply the first statement to $\conv V$ and see that an element of $V$ is in fact a minimum and thus $\EuclDistance(\Zonotope[\Directions],\cdot)$ is constant on $\conv V$. Since $\conv V$ contains an interior point of the minimal face $\Another\Cell$ of $\Cell$ that contains $V$, this shows that $\EuclDistance(\Zonotope[\Directions],\cdot)$ is constant on $\Another\Cell$.
\end{proof}

We record an elementary fact so that we can later on apply the proposition more easily:

\begin{obs}
\label{obs:sum_sufficiently_rich}
Let $\Cell_1$ and $\Cell_2$ be polytopes and let $\Cell$ be the convex hull of some of the vertices of $\Cell_1 \times \Cell_2$. Define $\pi \colon \EuclSpace \times \EuclSpace \to \EuclSpace$ by $\pi(x,y) = x-y$. If $\Directions_1$ is sufficiently rich for $\Cell_1$ and $\Directions_2$ is sufficiently rich for $\Cell_2$ then
\[
(\Directions_1 + \Directions_2) \union \Directions_1 \union \Directions_2 = ((\Directions_1 \union \{0\}) + (\Directions_2 \union \{0\})) \setminus \{0\}
\]
is sufficiently rich for $\pi(\Cell)$. In particular it is sufficiently rich for $\Cell_1 - \Cell_2$.
\end{obs}

\begin{proof}
Every vertex of $\pi(\Cell)$ is of the form $\Vertex_1 - \Vertex_2$ for vertices $\Vertex_i$ of $\Cell_i$. So if $\Vertex$ and $\AltVertex$ are distinct vertices of $\Cell$, then $\Vertex - \AltVertex = (\Vertex_1 - \AltVertex_1) + (\AltVertex_2 - \Vertex_2)$ where $\Vertex_i$ and $\AltVertex_i$ may or may not be distinct. If they are distinct for $i=1,2$, then $\Vertex - \AltVertex \in \Directions_1 + \Directions_2$. If $\AltVertex_i = \Vertex_i$ for some $i$, then $\Vertex - \AltVertex \in \Directions_{3-i}$.
\end{proof}

Now let $\Weyl$ be a finite reflection group of $\EuclSpace$ and assume that $\Zonotope$ is a $\Weyl$-invariant polytope (for example $\Zonotope[\Directions]$ for a $\Weyl$-invariant set $\Directions$). Then:

\begin{lem}
\label{lem:n_f_lemma}
Let $\Vector \in \EuclSpace$ be arbitrary and let $\NormalVector = \Vector - \Proj[\Zonotope](\Vector)$. Every $\Weyl$-chamber that contains $\Vector$ also contains $\NormalVector$.
\end{lem}

\begin{proof}
This is a reformulation of \cite[Lemma~2.1]{buxgrawit10}: Let $\FaceVector = \Proj[\Zonotope](\Vector)$ so that $\Vector = \FaceVector + \NormalVector$. By Observation~\ref{obs:space_decomposition} there is a face $\Face$ such that $\FaceVector \in \Face$ and $\NormalVector \in \NormalCone\Face$. Assume that there were a $\Weyl$-chamber $\Chamber$ that contains $\Vector$ but not $\NormalVector$. Let $\AltChamber$ be a chamber that contains $\NormalVector$ and let $\Wall$ be a wall that separates $\Chamber$ and $\AltChamber$. By \cite[Lemma~2.1]{buxgrawit10}, $\Wall$ cannot separate $\NormalVector$ and $\Vector$, hence $\Vector$ would have to lie inside $\Wall$. But then $\Wall$ would still separate $\FaceVector$ and $\NormalVector$ contradicting the same lemma.
\end{proof}

% ***************
% *** SECTION ***
% ***************

\section{Euclidean twin buildings}
\label{sec:euclidean_twin_buildings}

\subsection*{The metric codistance}

Throughout this section let $\EuclTwinBuilding$ be a Euclidean twin building (twin buildings are introduced in \cite[Section~5.8]{abrbro}). In \cite{buxgrawit10} a Euclidean codistance $\EuclCoDistance$ on $\EuclTwinBuilding$ was introduced: let $\TwinApartment$ be a twin apartment that contains $\PosPoint \in \PosEuclBuilding$ and $\NegPoint \in \NegEuclBuilding$. The opposition relation between $\PosApartment$ and $\NegApartment$ is bijective so we may isometrically identify both apartments with a Euclidean space $(\EuclSpace,\EuclDistance)$ so that the following diagram commutes:

\begin{diagram}
\label{dia:twin_identification}
 && \EuclSpace && \\
& \ruTo^{\iota_+} && \luTo^{\iota_-} &\\
\PosApartment && \rToFro^{\op} && \NegApartment \\
\end{diagram}

Note that this identification does not depend on the particular chosen twin apartment, because every other twin apartment that contains $\PosPoint$ and $\NegPoint$ is isometric to this one in a distance- and opposition-preserving way.

Using this setting, $\EuclCoDistance(\PosPoint,\NegPoint)$ is the Euclidean distance between $\iota_+(\PosPoint)$ and $\iota_-(\NegPoint)$. In other words, if we regard $\EuclSpace$ as a Euclidean vector space by picking an origin, then $\EuclCoDistance(\PosPoint,\NegPoint)$ is the length of the vector $\iota_+(\PosPoint) - \iota_-(\NegPoint)$.

If $\PosPoint$ and $\NegPoint$ are not opposite, then there is a geodesic ray $\Ray$ in $\EuclSpace$ issuing at $\iota_+(\PosPoint)$ and moving away from $\iota_-(\NegPoint)$. We denote the ray $\iota_+^{-1} \circ \Ray$ by $\InftyRay{\PosPoint}{\NegPoint}$. Analogously $\iota_- \circ \Another\Ray$, where $\Another\Ray$ is the ray that issues at  $\iota_-(\NegPoint)$ and moves away from $\iota_+(\PosPoint)$, is denoted by $\InftyRay{\NegPoint}{\PosPoint}$. The following is Proposition~1.4 in \cite{buxgrawit10}:

\begin{prop}
\label{prop:infty_ray_well_defined}
Assume that $\EuclCoDistance(\PosPoint,\NegPoint) > 0$. The rays $\InftyRay{\PosPoint}{\NegPoint}$ and $\InftyRay{\NegPoint}{\PosPoint}$ are well-defined geodesic rays in $\PosEuclBuilding$ respectively $\NegEuclBuilding$ (that is they are independent of the chosen twin apartment).
\end{prop}

\subsection*{Perturbed codistances}

As in \cite{buxgrawit10} we will have to perturb the codistance (for a rough motivation see the beginning of Section~\ref{sec:height_and_gradient}). In the rest of this section we will provide the tools to do so.

Let $\Weyl$ be the Weyl group of $\Infty\PosEuclBuilding$ which is the same as that of $\Infty\NegEuclBuilding$ (and is of type $\DummyType_n$ if $\EuclTwinBuilding$ is of type $\tilde{\DummyType}_n$). It acts in a natural way on $\EuclSpace$ (with chosen origin) and induces a chamber structure of simplicial fans on it. Let $\Zonotope$ be a convex, $\Weyl$-invariant polytope of $\EuclSpace$, that satisfies $-\Zonotope = \Zonotope$.

We define the $\Zonotope$-perturbed codistance to be
\[
\EuclCoDistance[\Zonotope](\PosPoint,\NegPoint) = \EuclDistance(\iota_+(\PosPoint) - \iota_-(\NegPoint),\Zonotope)
\]
where $\iota_+$ and $\iota_-$ are as in \eqref{dia:twin_identification}. It is clear that this is well-defined by the same arguments as before. We record the trivial but noteworthy fact (see Figure~\ref{fig:perturbed_codistance}):

\begin{obs}
\label{obs:perturbed_codistance_interpretations}
$\EuclCoDistance[\Zonotope](\PosPoint,\NegPoint) = \EuclDistance(\iota_+(\PosPoint),\iota_-(\NegPoint) + Z) = \EuclDistance(\iota_+(\PosPoint) + Z,\iota_-(\NegPoint))$.\qed
\end{obs}

\begin{figure}[!ht]
\begin{center}
\includegraphics{perturbed_codistance}
\end{center}
\caption{Each of the dashed lines has length $\EuclCoDistance[\Zonotope](\PosPoint,\NegPoint)$.}
\label{fig:perturbed_codistance}
\end{figure}

In the same spirit we define $\InftyRay{\PosPoint}{\NegPoint}[\Zonotope]$ for points $\PosPoint, \NegPoint$ with $\EuclCoDistance[\Zonotope](\PosPoint,\NegPoint) > 0$: It is $\iota_+^{-1} \circ \Ray$ where $\Ray$ is the geodesic ray in $\EuclSpace$ that issues at $\iota_+(\PosPoint)$ and moves away from (the projection point of $\iota_+(\PosPoint)$ onto) $\iota_-(\NegPoint) + \Zonotope$. The ray $\InftyRay{\NegPoint}{\PosPoint}[\Zonotope]$ is defined analogously. We have:

\begin{prop}
Assume that $\EuclCoDistance[\Zonotope](\PosPoint,\NegPoint) > 0$. The rays $\InftyRay{\PosPoint}{\NegPoint}[\Zonotope]$ and $\InftyRay{\NegPoint}{\PosPoint}[\Zonotope]$ are well-defined geodesic rays in $\PosEuclBuilding$ and $\NegEuclBuilding$ respectively (that is they are independent of the chosen twin apartment).
\end{prop}

\begin{proof}
We show the statement for $\InftyRay{\PosPoint}{\NegPoint}[\Zonotope]$. In fact we prove that the carrier of $\Infty{\InftyRay{\PosPoint}{\NegPoint}}$ contains $\Infty{\InftyRay{\PosPoint}{\NegPoint}[\Zonotope]}$. Since $\InftyRay{\PosPoint}{\NegPoint}$ is a well-defined ray in the building by Proposition~\ref{prop:infty_ray_well_defined}, this shows that $\Infty{\InftyRay{\PosPoint}{\NegPoint}[\Zonotope]}$ is a well-defined point of the building at infinity and the statement follows.

So we have to show that $\InftyRay{\PosPoint}{\NegPoint}[\Zonotope]$ is eventually contained in every sector, i.e.\ every $\Weyl$-chamber, in which $\InftyRay{\PosPoint}{\NegPoint}$ is eventually contained. To keep notation simple, let us take the origin of $\EuclSpace$ to be $\iota_-(\NegPoint)$. If $\Vector$ denotes the vector $\iota_+(\PosPoint)$ and $\NormalVector$ the vector $\Vector - \Proj[\Zonotope]\Vector$, then $\Vector$ is the direction of $\InftyRay{\PosPoint}{\NegPoint}$ and $\NormalVector$ is the direction of $\InftyRay{\PosPoint}{\NegPoint}[\Zonotope]$. And by Lemma \ref{lem:n_f_lemma} every $\Weyl$-chamber that contains $\Vector$ also contains $\NormalVector$.
\end{proof}

The following remark is independent from the rest of the article and the reader who gains no insight from it is advised to ignore it. It puts into perspective what we have said about Euclidean twin buildings above.

\begin{remnum}
Assume that $\EuclTwinBuilding$ is the twin building associated to $\GroupScheme(\FiniteField[q][t,t^{-1}])$. Then the group $\GroupScheme(\FiniteField[q](t))$ defines a system of apartments (in the sense of \cite[Definition~8.4]{weiss09}) $\SystemOfApartments_+$ on $\PosEuclBuilding$ as well as a system of apartments $\SystemOfApartments_-$ on $\NegEuclBuilding$. The elements of $\SystemOfApartments_+$ and $\SystemOfApartments_-$ are precisely the positive respectively negative halves of twin apartments (\cite[Section~6.12]{abrbro}). The building at infinity with respect to either of $\SystemOfApartments_+$ and $\SystemOfApartments_-$ is the spherical building $\SpherBuilding$ associated to $\GroupScheme(\FiniteField[q](t))$. In other words $\SpherBuilding$ naturally embeds in $\Infty\PosEuclBuilding$ as well as in $\Infty\NegEuclBuilding$ and thus can be embedded ``diagonally'' into $\Infty{(\PosEuclBuilding \times \NegEuclBuilding)} = \Infty\PosEuclBuilding * \Infty\NegEuclBuilding$.

For two points $\PosPoint \in \PosEuclBuilding$ and $\NegPoint \in \NegEuclBuilding$ the rays $\Infty{\InftyRay{\PosPoint}{\NegPoint}[\Zonotope]}$ and $\Infty{\InftyRay{\NegPoint}{\PosPoint}[\Zonotope]}$ each lie in $\SpherBuilding$ as embedded in the separate factors. But it is not hard to see that in fact they define the same point of $\SpherBuilding$ and so $1/\sqrt{2}\Infty{\InftyRay{\PosPoint}{\NegPoint}[\Zonotope]} + 1/\sqrt{2}\Infty{\InftyRay{\NegPoint}{\PosPoint}[\Zonotope]}$ lies in $\SpherBuilding$ as embedded diagonally.
\end{remnum}

% ***************
% *** SECTION ***
% ***************

\section{Height and gradient}
\label{sec:height_and_gradient}

We describe the strategy to prove Theorem~\ref{thm:main_theorem_geometric_formulation}. Let $\EuclTwinBuilding$ be an irreducible Euclidean twin building and let $\Group$ be a group that acts strongly transitively and with finite kernel on $\EuclBuilding$. Then $\Group$ acts on $\Space \defeq \PosEuclBuilding \times \NegEuclBuilding$ which is contractible, but the action is not cocompact (because for a point $(\PosPoint,\NegPoint)$ the codistance from $\PosPoint$ to $\NegPoint$ is preserved under the action of $\Group$). On the other hand the action on the subspace $\NullLevel' \defeq \{(\PosPoint,\NegPoint) \mid \PosPoint \op \NegPoint\}$ is cocompact but $\NullLevel'$ is not $(2n-2)$-connected. The rough idea is to use a ``thickened up'' version $\NullLevel$ of $\NullLevel'$ that is obtained by perturbing the codistance. We show that $\NullLevel$ is $(2n-2)$-connected by filtering $\Space$ by subspaces, the first of which is $\NullLevel$ and verifying that each subspace is obtained from the preceding by gluing in $2n$-cells up to homotopy. To make this work we have to apply the results from Sections~\ref{sec:horizontal_links} to \ref{sec:euclidean_twin_buildings}.

In this section we define the height function by which $\Space$ will be filtered and verify some of its properties. The height function will be refined to become a proper Morse function in Section~\ref{sec:horizontal_cells_and_subdivision}. In Section~\ref{sec:descending_links} the descending links are shown to be spherical or contractible. From this the Main Theorem is easily deduced in Section~\ref{sec:proof_of_main_theorem}.

Let $\EuclSpace$ be a Euclidean vector space. Let $\Zonotope$ be $\Zonotope[(\Directions + \Directions) \union \Directions]$ for some $\Weyl$-invariant finite subset $\Directions$ of $\EuclSpace$ such that $\Directions = -\Directions$ and define the \emph{height} to be $\Height = \EuclCoDistance[\Zonotope]$. The height will play an important role in the rest of the article; note that it depends on $\Zonotope$ and thus on $\Directions$. For the first statement all that is needed is that $\Zonotope$ is convex.

\begin{obs}
\label{obs:height_convex_on_apartments}
If $\TwinApartment$ is a twin apartment, then the restriction of $\Height$ to $\PosApartment \times \NegApartment$ is a convex function. In particular $\Height$ attains its maximum over a cell in a vertex. And the intersection of a cell with an $\Height$-sublevel set is convex.
\end{obs}

\begin{proof}
We make identifications as in \eqref{dia:twin_identification}. By definition the restriction of $\Height$ to $\PosApartment \times \NegApartment$ is the composition of the map $(\PosPoint,\NegPoint) \mapsto \iota_+(\PosPoint) - \iota_-(\NegPoint)$ which is affine (that is it preserves affine combinations) with the convex map $\EuclDistance(\Zonotope,\cdot)$ and thus convex.
\end{proof}

\begin{figure}[!ht]
\begin{center}
\includegraphics{zonotope_convolution}
\end{center}
\caption{The set $\NullLevel$ as seen inside $\PosApartment \times \NegApartment$.}
\label{fig:zonotope_convolution}
\end{figure}

Now we want to define a gradient for $\Height$. To do so we first have to understand the height function itself a little better. It would be reasonable to expect it to measure distance from the set $\NullLevel \defeq \Height^{-1}(0)$ of level $0$. We will convince ourselves that this is true up to rescaling. Let $\TwinApartment$ be a twin apartment and make the identifications as in \eqref{dia:twin_identification}. For a point $\Point = (\PosPoint,\NegPoint) \in \PosApartment \times \NegApartment$ with $\Height(\Point) > 0$ we consider the sets
\[
\Zonotope_+ \defeq \big(\opm_{\TwinApartment} \NegPoint\big) + \iota_+^{-1}(\Zonotope)\quad \text{and} \quad \Zonotope_- \defeq  \big(\opm_{\TwinApartment} \PosPoint\big) + \iota_-^{-1}(\Zonotope)\text{ .}
\]
By Observation~\ref{obs:perturbed_codistance_interpretations} $\Height(\Point)$ can be seen as the distance from $\PosPoint$ to $\Zonotope_+$ or from $\NegPoint$ to $\Zonotope_-$ (see Figure~\ref{fig:zonotope_convolution}). Note that $\Zonotope_+ \times \{\NegPoint\}$ is just the intersection of the slice $\PosApartment \times \{\NegPoint\}$ with $\NullLevel$ so that $(\Proj[\Zonotope_+]\PosPoint,\NegPoint) \in \NullLevel$ and the same is true with signs reversed. However $\Proj[\NullLevel]\Point = (1/2(\PosPoint + \Proj[\Zonotope_+]\PosPoint),1/2(\NegPoint + \Proj[\Zonotope_-]\NegPoint))$ and thus $\Height(\Point) = \sqrt{2}\EuclDistance(\Point,\NullLevel)$.

This shows that the geodesic ray away from $\NullLevel$, that is the ray in the direction in which $\Height$ grows fastest, is the ray that has components $\InftyRay{\PosPoint}{\NegPoint}[\Zonotope]$ and $\InftyRay{\NegPoint}{\PosPoint}[\Zonotope]$ as introduced in Section~\ref{sec:euclidean_twin_buildings}. Normalized to unit-speed it is given by
\[
t \mapsto \bigg(\InftyRay{\PosPoint}{\NegPoint}[\Zonotope]\Big(\frac{1}{\sqrt{2}}t\Big),\InftyRay{\NegPoint}{\PosPoint}[\Zonotope]\Big(\frac{1}{\sqrt{2}}t\Big)\bigg) \text{ .}
\]
We denote the point at infinity defined by this ray by $\InftyGradient_\Point\Height$ and call it the \emph{asymptotic gradient of $\Height$ at $\Point$}. The direction defined by the ray is called the \emph{gradient of $\Height$ at $\Point$} and denoted $\Gradient_\Point\Height$. Clearly $(\InftyGradient_\Point\Height)_\Point = \Gradient_\Point\Height$.

Recall that according to the definition in Section~\ref{sec:horizontal_links} a point $\PointAtInfty \in \Infty\Space$ is in general position if it is neither in $\Infty\PosEuclBuilding$ nor in $\Infty\NegEuclBuilding$. Since both components of  $\InftyGradient_\Point\Height$ are non-trivial, we get:

\begin{obs}
\label{obs:gradient_in_general_position}
The asymptotic gradient $\InftyGradient_\Point\Height$ of $\Height$ at any point $\Point$ is in general position.\qed
\end{obs}

Calling $\Gradient\Height$ a gradient is justified by the following angle criterion:

\begin{obs}
\label{obs:infinitesimal_angle_criterion}
Let $\Geodesic$ be a geodesic that is contained in a cell and issues at a point $\Point$ of positive height. The function $\Height \circ \Geodesic$ is strictly increasing if $\angle_\Point(\Gradient_\Point\Height,\Geodesic) < \pi/2$. It is strictly decreasing on an initial interval if $\angle_\Point(\Gradient_\Point\Height,\Geodesic) > \pi/2$.

In particular if $\Gradient_\Point\Height$ is perpendicular to a cell that contains $\Point$, then $\Point$ is a point of minimal height of that cell.
\end{obs}

\begin{proof}
Let $\Cell = \PosCell \times \NegCell$ be a cell that contains $\Geodesic$. Let $\TwinApartment$ be a twin apartment that contains $\PosCell$ and $\NegCell$. The monotonicity statements on an initial interval now follow from the above discussion applied to the Euclidean space $\PosApartment \times \NegApartment$. That $\Geodesic$ is globally increasing if $\angle_\Point(\Geodesic,\Gradient_\Point\Height) < \pi/2$ then follows from the fact that $\Height$ is convex on $\PosApartment \times \NegApartment$ (Observation~\ref{obs:height_convex_on_apartments}).

The last statement follows from the fact that all geodesics from $\Point$ to another point of $\Cell$ are non-decreasing which follows by continuity from the first statement.
\end{proof}

We say that $\Directions$ is \emph{almost rich} if it contains $\iota_+^{-1}(\Vertex_1) - \iota_+^{-1}(\Vertex_2)$ for any two adjacent vertices $\Vertex_1$ and $\Vertex_2$ of an apartment $\PosApartment$ of $\PosEuclBuilding$ with the identifications in \eqref{dia:twin_identification} (the analogous statement for an apartment of $\NegEuclBuilding$ is then automatic). The point in perturbing the distance is to render the angle criterion true on a macroscopic scale:

\begin{prop}
\label{prop:almost_rich_implies_minimum_in_vertex}
Assume that $\Directions$ is almost rich and let $\Cell$ be a cell. Then $\Height$ attains its minimum over $\Cell$ in a vertex and the set of $\Height$-maxima of $\Cell$ is a face.

The statement remains true with $\Cell$ replaced by the convex hull of some of its vertices.
\end{prop}

\begin{proof}
Let $\Cell = \PosCell \times \NegCell$ be the cell and let $\TwinApartment$ be a twin apartment that contains $\PosCell$ and $\NegCell$. We make the identifications as in \eqref{dia:twin_identification}. Since $\Directions$ is almost rich, it is sufficiently rich for $\iota_+(\PosCell)$ as well as for $\iota_-(\NegCell)$, so that $(\Directions + \Directions) \union \Directions$ is sufficiently rich for $\iota_+(\PosCell) - \iota_-(\NegCell)$ by Observation~\ref{obs:sum_sufficiently_rich}. Thus Proposition~\ref{prop:sufficiently_rich_min_vertex} shows that $\EuclDistance(\Zonotope,\cdot)$ attains its minimum over $\PosCell - \NegCell$ in a vertex and the set of maxima is a face. The map $\Height$ is the composition of this function with the affine map $(\PosPoint,\NegPoint) \mapsto \iota_+(\PosPoint) - \iota_-(\NegPoint)$. And inverse images of faces under affine maps are again faces.

The second statement is shown analogously.
\end{proof}

\begin{cor}
\label{cor:monotone_on_edges}
Assume that $\Directions$ is almost rich. Let $\Vertex$ and $\AltVertex$ be vertices that are contained in a common cell. Then $\Height$ is monotone on $[\Vertex,\AltVertex]$. In particular $\Height(\Vertex) > \Height(\AltVertex)$ if and only if $\angle_\Vertex(\Gradient_\Vertex \Height,\AltVertex) > \pi/2$.
\end{cor}

\begin{proof}
Monotonicity follows from convexity together with the fact that $\Height$ attains its minimum in a vertex by Proposition~\ref{prop:almost_rich_implies_minimum_in_vertex}. The second statement then follows from Observation~\ref{obs:infinitesimal_angle_criterion}.
\end{proof}

% ***************
% *** SECTION ***
% ***************

\section{Horizontal cells and subdivision}
\label{sec:horizontal_cells_and_subdivision}

From now on we assume $\Directions$ to be almost rich. A cell $\Cell$ of $\Space$ is called \emph{horizontal} if $\Height|_{\Cell}$ is constant.

\begin{obs}
\label{obs:gradient_independent_of_maximal_point}
For a horizontal cell $\Cell$, the asymptotic gradient $\InftyGradient_\Point\Height$ is the same at all points $\Point \in \Cell$.
\end{obs}

\begin{proof}
Write $\Cell = \PosCell \times \NegCell$ and let $\TwinApartment$ be a twin apartment that contains $\PosCell$ and $\NegCell$. Let $\EuclDistance$ denote the metric on $\PosApartment \times \NegApartment$. Assume that there are two points $\Point$ and $\AltPoint$ in $\Cell$ that have non-identical asymptotic gradients, i.e.\ non-parallel gradients. Let $m = \EuclDistance(\Point,\NullLevel) = \EuclDistance(\AltPoint,\NullLevel)$. Let $\Vector$ be a unit vector that points in the direction of $\Gradient_\Point\Height$ and $\AltVector$ be a unit vector that points in the direction of $\Gradient_\AltPoint\Height$. The projection of $\Point$ respectively $\AltPoint$ to $\NullLevel$ is $\Proj[\NullLevel] \Point  = \Point - m\Vector$ respectively $\Proj[\NullLevel] \AltPoint  = \AltPoint - m\AltVector$. So $\EuclDistance(1/2 \Point + 1/2 \AltPoint,1/2 \Proj[\NullLevel]\Point + 1/2 \Proj[\NullLevel]\AltPoint) = m \lVert 1/2\Vertex + 1/2\AltVertex \rVert < m$ contradicting that the set of maxima is convex by Proposition~\ref{prop:almost_rich_implies_minimum_in_vertex}.
\end{proof}

The Observation allows us to define (asymptotic) gradients for cells: if $\Cell$ is horizontal, we set $\InftyGradient_\Cell\Height$ to be $\InftyGradient_\Point\Height$ for any point $\Point$ of maximal height of $\Cell$. Similarly $\Gradient_\Cell\Height \defeq \Gradient_\Point\Height$ which is an element of $\Link \Cell$ because $\Cell$ is horizontal. We take this direction to be the north pole in $\Link \Cell$ (cf.\ Section~\ref{sec:spherical_subcomplexes}) and accordingly define $\OpenHemisphere{\Link}\Cell$, $\ClosedHemisphere{\Link}\Cell$, $\Horizontal{\Link}\Cell$, and $\Vertical{\Link}\Cell$. The decomposition \eqref{eq:horizontal_vertical_decomposition} then reads
\begin{equation}
\Link \Cell = \Horizontal{\Link} \Cell * \Vertical{\Link} \Cell \text{ .}
\end{equation}
We call $\Horizontal{\Link} \Cell$ the \emph{horizontal link} and $\Vertical{\Link} \Cell$ the \emph{vertical link} of $\Cell$.

If $\BigCell$ is a horizontal cell then $\Height$ fails to be a Morse function on $\BigCell$. Locally $\Height$ resembles a Busemann function that defines $\InftyGradient_\BigCell\Height$. By Observation~\ref{obs:gradient_in_general_position} this Busemann function is in general position and we can apply the machinery from Section~\ref{sec:horizontal_links} to define a secondary Morse function on the barycentric subdivision: we define the \emph{depth of $\BigCell$}, denoted $\Depth \BigCell$, to be the length of the longest sequence of moves (counting moves, not cells) with respect to $\InftyGradient_\BigCell\Height$ that starts with $\BigCell$. Proposition~\ref{prop:bound_on_moves} implies that $\Depth\BigCell$ is well-defined. We also define $\BigCell\Min$ to be the unique minimal face $\Cell$ of $\BigCell$ such that $\BigCell \horizontal_{\InftyGradient_{\BigCell}\Height} \Cell$ which exists by Lemma~\ref{lem:tau_min}. If $\BigCell\Min = \BigCell$ then $\BigCell$ is called \emph{essential}.

If $\Cell$ is an arbitrary cell, we define the \emph{roof} $\Roof\Cell$ of $\Cell$ to be the face of points of maximal height (see Proposition~\ref{prop:almost_rich_implies_minimum_in_vertex}). So horizontal cells are their own roof. If $\Cell$ is non-horizontal, we set $\Depth \Cell = \Depth \Roof\Cell - 1/2$.

Now we are ready to define our final Morse function. Let $\Subdivision\Space$ be the barycentric subdivision of $\Space$. We define $\Subdivision\Height$ on the vertex set of $\Subdivision\Space$ to be
\begin{eqnarray*}
\Vertices\Subdivision\Height \colon \Subdivision\Space & \to & \real \times \tfrac{1}{2}\integer \times \nat \\
\Subdivision\Cell & \mapsto & (\max h|_\Cell,\Depth(\Cell),\dim(\Cell))
\end{eqnarray*}
where the range is ordered lexicographically.

Recall that the cells of $\Subdivision\Space$ correspond to flags of cells of $\Space$. In particular adjacent vertices correspond to cells one of which is contained in the other, assuring that both have different dimension. Thus we have:

\begin{obs}
The function $\Subdivision\Height$ is a Morse function in the sense that if $\Subdivision\Cell$ and $\Subdivision\BigCell$ are adjacent vertices, then $\Subdivision\Height(\Subdivision\Cell) \ne \Subdivision\Height(\Subdivision\BigCell)$.\qed
\end{obs}

According to \eqref{eq:link_of_point_decomposition} the link of $\Subdivision\Cell$ decomposes as $\Link \Subdivision \Cell = \FacePart_{\Subdivision\Cell} \Cell * \Link \Cell$. But the barycentric subdivision of the ambient space also induces the barycentric subdivision of each of the factors which we take into account notationally by writing
\begin{equation}
\label{eq:link_of_barycenter_decomposition}
\Link \Subdivision\Cell = \Link[\FacePart]\Subdivision\Cell * \Link[\CofacePart] \Subdivision\Cell
\end{equation}
and calling $\Link[\FacePart]\Subdivision\Cell$ the \emph{face part} and $\Link[\CofacePart] \Subdivision\Cell$ the \emph{coface part} of $\Link\Subdivision\Cell$.

% ***************
% *** SECTION ***
% ***************

\section{Descending links}
\label{sec:descending_links}

In what follows we identify the combinatorial link of $\Subdivision\Cell$ with the geometric link. In particular a vertex adjacent to $\Subdivision\Cell$ is identified with the vertex of $\Link\Subdivision\Cell$ that it defines. For a vertex $\Subdivision\Cell$ of $\Subdivision\Space$ we define the \emph{descending link} $\DescendingLink\Subdivision\Cell$ to be the full subcomplex of $\Link\Subdivision\Cell$ of vertices $\Subdivision\BigCell$ with $\Subdivision\Height(\Subdivision\BigCell) < \Subdivision\Height(\Subdivision\Cell)$. Since $\Link\Subdivision\Cell$ is a flag complex, \eqref{eq:link_of_barycenter_decomposition} induces a decomposition
\begin{equation}
\label{eq:descending_link_of_barycenter_decomposition}
\DescendingLink \Subdivision\Cell = \DescendingLink[\FacePart]\Subdivision\Cell * \DescendingLink[\CofacePart] \Subdivision\Cell \text{ ,}
\end{equation}
where of course $\DescendingLink[\FacePart]\Subdivision\Cell \defeq \Link[\FacePart]\Subdivision\Cell \intersect \DescendingLink \Subdivision\Cell$ and $\DescendingLink[\CofacePart] \Subdivision\Cell \defeq \Link[\CofacePart] \Subdivision\Cell \intersect \DescendingLink \Subdivision\Cell$.

Our goal in this section is to show that $\DescendingLink[\FacePart]\Subdivision\Cell$ is $(\dim \Space - 2)$-connected.

\begin{lem}
\label{lem:descending_link_of_non-essential_cell}
If $\BigCell$ is horizontal and not essential ($\BigCell\Min \ne \BigCell$) then the descending link of $\Subdivision\BigCell$ is contractible.
\end{lem}

\begin{proof}
Again we prove that $\DescendingLink[\FacePart]\Subdivision\BigCell$ is contractible. First note that all faces $\Cell$ of $\BigCell$ satisfy $\Height(\Cell) = \Height(\BigCell)$ and $\InftyGradient_{\Subdivision\Cell}\Height = \InftyGradient_{\Subdivision\BigCell}\Height$, so a move $\Cell \Up \BigCell$ implies $\Depth\Cell > \Depth\BigCell$ and a move $\BigCell \Down \Cell$ implies $\Depth\BigCell > \Depth\Cell$.

Now if a proper face $\Cell$ does not contain $\BigCell\Min$, then $\BigCell \Down \Cell$. So $\Subdivision\Cell$ is descending for all these cells $\Cell$. On the other hand $\BigCell\Min \Up \BigCell$ so $\Subdivision{\BigCell\Min}$ is not descending. Even though it is not clear whether $\Subdivision\Cell$ is descending or not if $\BigCell\Min \lneq \Cell \lneq \BigCell$, we can deduce that $\DescendingLink[\FacePart]\Subdivision\BigCell$ is the sphere $\Link[\FacePart]\Subdivision\BigCell$ punctured at the point $\Subdivision{\BigCell\Min}$ and hence contractible.
\end{proof}

\begin{lem}
\label{lem:descending_link_of_non-horizontal_cell}
If $\BigCell$ is not horizontal then the descending link of $\Subdivision\BigCell$ is contractible.
\end{lem}

\begin{proof}
We show that $\DescendingLink[\FacePart]\Subdivision\BigCell$ is contractible. If $\Cell$ is a face of $\BigCell$ then clearly $\dim \Cell < \dim \BigCell$ so for $\Subdivision\Cell$ to be non-descending necessarily $\max \Height|_\Cell = \max \Height|_\BigCell$ and $\Depth \Cell < \Depth \BigCell$. We have three classes of proper faces of $\BigCell$: the first class consists of cells $\Cell$ with $\max \Height|_\Cell < \max \Height|_\BigCell$. The second class consists of non-horizontal faces $\Cell$ with $\Roof\Cell \le \Roof\BigCell$. And the third class consists of faces of $\Roof\BigCell$.

For the cells of the first class $\Subdivision\Cell$ is always descending and $\Subdivision{\Roof\BigCell}$ is never descending (because by definition $\Depth\BigCell = \Depth\Roof\BigCell - 1/2$). For the remaining faces, $\Subdivision\Cell$ is descending if and only if $\Depth\Roof\Cell < \Depth\Roof\BigCell$. So there is an easy case: if $\Roof\BigCell$ is essential, then there is a move down $\Roof\BigCell \Down \Cell$ to all faces $\Cell$ of $\Roof\BigCell$. This means that $\Subdivision\Cell$ is descending for all $\Cell \le \BigCell$ except for $\Cell = \Roof\BigCell$. Thus $\DescendingLink[\FacePart]\Subdivision\BigCell$ is the sphere $\Link[\FacePart]\Subdivision\BigCell$ punctured at $\Subdivision{\Roof\BigCell}$ and hence contractible.

If $\Roof\BigCell$ is not essential, its descending link is punctured by Lemma~\ref{lem:descending_link_of_non-essential_cell}. In this case there is a deformation retraction that takes $\Subdivision\Cell$ to $\Subdivision{\Roof\Cell}$ for cells $\Cell$ of the second class with $\Depth\Roof\Cell < \Depth\Roof\BigCell$. And there is a deformation retraction that takes $\Subdivision\Cell$ to $\Subdivision{\Roof\BigCell}$ for cells $\Cell$ of the third class with $\Depth\Cell < \Depth\Roof\BigCell$. So the full subcomplex of vertices of $\Link[\FacePart]\Subdivision\BigCell$ not in $\DescendingLink[\FacePart]\Subdivision\BigCell$ is contractible and thus its complement is also contractible.
%(by the combinatorial Jordan-Schoenflies Theorem, see \cite[Theorem~5]{newman60} if necessary).
So Proposition~\ref{prop:retract_set_onto_subcomplex} implies that $\DescendingLink[\FacePart]\Subdivision\BigCell$ is contractible.
\end{proof}

\begin{lem}
\label{lem:descending_link_of_essential_cell_face_part}
If $\Cell$ is essential, then all of $\Link[\FacePart] \Subdivision\Cell$ is descending.
\end{lem}

\begin{proof}
Let $\Another\Cell$ be a proper face of $\Cell$. Clearly $\Height|_\Cell = \Height|_{\Another\Cell}$. And $\Another\Cell \lneq \Cell = \Cell\Min$ so that there is a move $\Cell \Down \Another\Cell$ which implies $\Depth \Cell > \Depth \Another\Cell$ so that $\Another\Cell$ is descending.
\end{proof}

Recall from \eqref{eq:link_of_barycenter_decomposition} that the coface part of the link of $\Subdivision\Cell$ in $\Subdivision\Space$ is the barycentric subdivision of the link of $\Cell$ in $\Space$: $\Link[\CofacePart] \Subdivision\Cell$ is the barycentric subdivision of $\Link \Cell$.

\begin{prop}
\label{prop:essential_coface_links_dont_subdivide}
If $\Cell$ is essential ($\Cell = \Cell\Min$), then the coface part of the descending link of $\Cell$ is a subcomplex of $\Link \Cell$.
\end{prop}

\begin{proof}
We have to show that if $\BigCell \ge \Cell$ is a cell that is descending, i.e.\ $\Subdivision\Height(\Subdivision\BigCell) < \Subdivision\Height(\Subdivision\Cell)$, then $\Another\Cell$ is also descending for $\Cell \lneq \Another\Cell \le \BigCell$.

So let $\BigCell \ge \Cell$ be a descending coface of $\Cell$. First observe that $\max \Height|_\BigCell \ge \max \Height|_\Cell$ so for $\BigCell$ to be descending it is necessary that actually equality holds. Also $\dim \BigCell > \dim \Cell$, so $\BigCell$ is descending if and only if $\Depth \BigCell < \Depth \Cell$.

It is possible that $\Roof\BigCell = \Cell$ so that $\Depth \BigCell = \Depth \Cell - 1/2$. Otherwise $\Roof\BigCell$ is a proper coface of $\Cell$ so that by Observation~\ref{obs:essential_either_or} there has to be a move either $\Cell \Up \Roof\BigCell$ or $\Roof\BigCell \Down \Cell$. The second possibility of move however would mean that $\Depth \Roof\BigCell \ge \Depth \Cell + 1$ so that $\Depth \BigCell > \Cell$ which is not the case. Hence we see that there is a move $\Cell \Up \Roof\BigCell$, i.e.\ ${\Roof\BigCell}\Min = \Cell$.

Now let $\Another\Cell$ be a face of $\BigCell$ that properly contains $\Cell$. We have $\max \Height|_\BigCell \ge \max \Height|_{\Another\Cell} \ge \max \Height|_\Cell$ and since the first and the third term are equal, all three are equal. If $\Another{\Roof\Cell} = \Cell$, then $\Another\Cell$ is descending and we are done. Otherwise $\Another{\Roof\Cell}$ is a face of $\Roof\BigCell$ that properly contains $\Cell$. Then, since $\Roof{\BigCell}\Min = \Cell$, Observation~\ref{obs:min_min} implies that $\Another{\Roof\Cell}{}\Min = \Cell$ so that there is a move $\Cell \Up \Another{\Roof\Cell}$. Thus $\Depth \Another\Cell \le \Depth \Another{\Roof\Cell} < \Depth \Cell$ and $\Another\Cell$ is descending.
\end{proof}

The importance of the proposition is the following: the barycentric subdivision of $\Link \Cell = \Horizontal\Link \Cell * \Vertical\Link \Cell$ does \emph{not} decompose as a join of the subdivisions because there are cells with faces in both join factors that are subdivided. So a priori the analysis of the descending link can not be separated in horizontal and vertical part. However the proposition says that if $\Cell$ is essential, then we can regard the coface part of the descending link as a subcomplex of $\Link \Cell$ and thus write
\begin{equation}
\DescendingLink \Subdivision\Cell = \Descending{\Link[\FacePart]} \Subdivision\Cell * \Descending{\Horizontal\Link} \Cell * \Descending{\Vertical\Link} \Cell
\end{equation}
where $\Descending{\Horizontal\Link} \Cell$ and $\Descending{\Vertical\Link} \Cell$ are the subcomplexes of $\Horizontal\Link \Cell$ respectively  $\Vertical\Link \Cell$ of cells $\BigCell \direction \Cell$ such that $\BigCell$ is descending. It remains to analyze the separate join factors.

\begin{lem}
\label{lem:descending_link_of_essential_cell_vertical_coface_part}
If $\Cell$ is essential, then $\Descending{\Vertical\Link} \Cell$ is the open hemisphere complex $\OpenHemisphere\Link \Cell$ with north pole $\Gradient_\Cell\Height$.
\end{lem}

\begin{proof}
Let $\BigCell$ be a coface of $\Cell$ such that $\BigCell \direction \Cell$ lies in the vertical link.

If $\max \Height|_\BigCell > \max \Height|_\Cell$, then $\BigCell \direction \Cell$ is not contained in $\Descending{\Vertical\Link} \Cell$. Let $\Vertex$ be a vertex of $\Cell$ and $\AltVertex$ be a vertex of $\BigCell$ with $\Height(\AltVertex) > \Height(\Vertex)$ (which exists by Observation~\ref{obs:height_convex_on_apartments}). By Corollary~\ref{cor:monotone_on_edges} $\angle_\Vertex(\AltVertex,\Gradient_\Vertex\Height) \le \pi/2$. This implies that $\BigCell$ is not contained in $\OpenHemisphere\Link \Cell$.

If $\max \Height|_\BigCell = \max \Height|_\Cell$, we distinguish two cases: In the first case $\Roof\BigCell = \Cell$ so that $\Depth\BigCell = \Depth\Cell - 1/2$ and $\BigCell \direction \Cell$ lies in $\Descending{\Vertical\Link} \Cell$. Then $\Height(\Vertex) < \Height(\Cell)$ for every vertex $\Vertex$ of $\BigCell$ not contained in $\Cell$. Using again Corollary~\ref{cor:monotone_on_edges} this implies that $\BigCell \direction \Cell$ is contained in $\OpenHemisphere\Link \Cell$.

In the second case $\Roof\BigCell$ is a proper coface of $\Cell$. Since $\Roof\BigCell$ lies in the vertical link of $\Cell$, there is a move $\Roof\BigCell \Down \Cell$ so that $\Depth \BigCell = \Depth \Roof\BigCell -1/2 \ge \Depth \Cell + 1/2$ and $\BigCell \direction \Cell$ does not lie in $\Descending{\Vertical\Link} \Cell$. Let $\Vertex$ be a vertex of $\Roof\BigCell$ that does not lie in $\Cell$. Then $\Height(\Vertex) = \Height(\Cell)$, so Corollary~\ref{cor:monotone_on_edges} implies that $\BigCell \direction \Cell$ does not lie in $\OpenHemisphere\Link \Cell$ either.
\end{proof}

For the horizontal links of essential cells to be spherical, we have to impose a stronger condition on $\Directions$ than merely being almost rich: we call $\Directions$ \emph{rich} if it contains $\iota_+(\Vertex) - \iota_+(\AltVertex)$ for any two distinct vertices $\Vertex$ and $\AltVertex$ that are contained in a common (closed) star (i.e.\ can be joined by a path of up to two edges). We assume from now on that $\Directions$ is rich.

When we try to analyze the horizontal link of an essential cell $\Cell = \PosCell \times \NegCell$, we have to deal with another obstacle that is new compared to the situation of \cite{buxgrawit10}. Eventually we want to apply Proposition~\ref{prop:apartmentwise_coconvex_complexes}, that is we have to find a chamber $\Chamber$ in $\Link\Cell$ such that for every apartment that contains $\Chamber$ we know which simplices of the apartment are descending and which are not. The apartments that we understand well are apartments of the form $(\PosApartment \intersect \Link\PosCell) * (\NegApartment \intersect \Link\NegCell)$ where $\TwinApartment$ is a twin apartment. In the generic case the projection of $\NegCell$ onto $\Star\PosCell$ is a chamber $\PosChamber$ and the projection of $\PosCell$ onto $\Star\NegCell$ is a chamber $\NegChamber$. If this is the case, then we can just take $\Chamber$ to be $(\PosChamber \direction \PosCell) * (\NegChamber \direction \NegCell)$ because every twin apartment that contains $\PosCell$ and $\NegCell$ also contains $\PosChamber$ and $\NegChamber$ and every apartment of $\Link\Cell$ that contains $\Chamber$ comes from a twin apartment. In general however the projections may have positive codimension and for these cases, we have to understand a greater class of apartments.

Note that if two chambers $\NegChamber, \Another\NegChamber \ge \NegCell$ project to distinct chambers in $\Star\PosCell$, then there has to be a wall in the twin building that separates these chambers as well as their projections. That is, $\PosCell$ and $\NegCell$ lie in a common wall. This is the situation of the following lemma. Recall that we assumed $\Directions$ to be rich.

\begin{lem}
\label{lem:reflect_halve_twin_preserves_height}
Let $\TwinApartment$ be a twin apartment, let $\PosCell \subseteq \PosApartment$ and $\NegCell \subseteq \NegApartment$ be cells. Assume that there is wall $\Wall$ that contains $\PosCell$ and $\NegCell$. Then the automorphism of $(\Star\PosCell \intersect \PosApartment) \times (\Star\NegCell \intersect \NegApartment)$ that is given by fixing $\NegApartment$ and reflecting $\PosApartment$ at $\Wall$ preserves $\Height$. The same is true with signs reversed.
\end{lem}

\begin{proof}
Make identifications as in \eqref{dia:twin_identification} and let $\EuclSpace$ be equipped with the Coxeter complex structure from $\PosApartment$ which is the same as that from $\NegApartment$. In particular $\Wall$ is a wall in $\EuclSpace$. Let $r_\Wall$ denote the reflection of $\EuclSpace$ at $\Wall$. Let $E \defeq (D+D) \union D$.

Let $\PosVertex \in \Star\PosCell$ and $\NegVertex \in \Star\NegCell$ be vertices. We will show that the segment $[\PosVertex,r_\Wall(\PosVertex)]$ is linearly projected onto $-\NegVertex + \Zonotope(E)$. From this the result follows since arbitrary points are convex combinations of vertices. To simplify notation we take the origin of $\EuclSpace$ to be $\Proj[\Wall](-\NegVertex)$.

Let $\PosPoint = \Proj[\Wall]\PosVertex$. Note that $\Zonotope(E) = \Zonotope(E \intersect \Wall^\perp) + \Zonotope(E \setminus \Wall^\perp)$ and that $\Proj[\Zonotope(E)]\PosPoint \in \Zonotope(E \setminus \Wall^\perp)$.

The vector $2\NegVertex = \NegVertex - r_\Wall(\NegVertex)$ is an element of $\Directions$ by richness and it is clearly perpendicular to $\Wall$. Hence $\Proj[\Zonotope(E)]\PosPoint = \Proj[-\NegVertex + \Zonotope(E)]\PosPoint$ and $\Proj[-\NegVertex + \Zonotope(E)]\PosPoint + \Zonotope(D \intersect \Wall^\perp) \subseteq -\NegVertex + \Zonotope(E)$. Also by richness we have $2(\PosVertex - \PosPoint)$, so that $\PosVertex$ and $r_\Wall(\PosVertex)$ both lie in $\PosPoint + \Zonotope(D \intersect \Wall^\perp)$.
\end{proof}

\begin{cor}
\label{cor:group_of_reflections_at_coprojection_preserves_height}
Let $\TwinApartment$ be a twin apartment, let $\PosCell \subseteq \PosApartment$ and $\NegCell \subseteq \NegApartment$ be cells and let $\PosBigCell = \Proj[\Star\PosCell] \NegCell$ and $\NegBigCell = \Proj[\Star\NegCell] \PosCell$. The Coxeter group $\Stabilizer{\Aut(\PosApartment \intersect \Star\PosCell)}{\PosBigCell} \times \Stabilizer{\Aut(\NegApartment \intersect \Star\NegCell)}{\NegBigCell}$ preserves $\Height$.
\end{cor}

\begin{proof}
The group is generated by the reflections described in the Lemma \ref{lem:reflect_halve_twin_preserves_height}.
\end{proof}

\begin{prop}
Let $\PosCell \subseteq \PosEuclBuilding$ and $\NegCell \subseteq \NegEuclBuilding$ be cells and let $\PosBigCell = \Proj[\Star\PosCell] \NegCell$ and $\NegBigCell = \Proj[\Star\NegCell] \PosCell$. The group $L \defeq \Stabilizer{\Aut\Star\PosCell}{\PosBigCell} \times \Stabilizer{\Aut\Star\NegCell}{\NegBigCell}$ preserves $\Height$.
\end{prop}

\begin{proof}
It suffices of course to consider one factor of $L$.
Let $\PosChamber \ge \PosCell$ and $\NegChamber \ge \NegCell$ be chambers and let $g \in \Stabilizer{\Aut\Star\PosCell}{\PosBigCell}$. Let $\TwinApartment$ be a twin apartment that contains $\PosChamber$ and $\NegChamber$ and let $\Retraction$ be the retraction onto $\TwinApartment$ centered at $\NegChamber$. Let $\AltChamber = \Proj[\Star\PosCell]\NegChamber$ which contains $\PosBigCell$. Recall that $\rho(g\PosChamber)$ is the chamber in $\TwinApartment$ that has Weyl-codistance $\WCoDistance(\NegChamber,g\PosChamber)$ from $\NegChamber$. Equivalently it is the chamber in $\PosApartment$ that has Weyl-distance $\PosWDistance(\AltChamber,g\PosChamber)$ from $\AltChamber$. We compute
\begin{eqnarray*}
\WCoDistance(\NegChamber,\PosChamber)\WCoDistance(\NegChamber,g\PosChamber)^{-1} & = & \PosWDistance(\AltChamber,\PosChamber)\PosWDistance(\AltChamber,g\PosChamber)^{-1} \\
& = & \PosWDistance(\AltChamber,\PosChamber)\PosWDistance(g^{-1}\AltChamber,\PosChamber)^{-1}
\end{eqnarray*}
which is an element of the Weyl group of $\Star\PosBigCell$. Thus $\Retraction(g\PosChamber)$ lies in the orbit of $\PosChamber$ under $\Stabilizer{\Aut(\PosApartment\intersect\Star\PosCell)}{\PosBigCell}$ and so $\Height(\Point) = \Height(g\Point)$ for every $\Point \in \PosChamber$ by Corollary~\ref{cor:group_of_reflections_at_coprojection_preserves_height}.
\end{proof}

\begin{cor}
\label{cor:apartments_that_contain_chamber_look_good}
Let $\PosCell \subseteq \PosEuclBuilding$ and $\NegCell \subseteq \NegEuclBuilding$ be cells and let $\Chamber_+ \ge \PosCell$, $\Chamber_- \ge \NegCell$ be chambers such that $\Proj[\Star\PosCell]\Chamber_- = \Chamber_+$ and $\Proj[\Star\NegCell]\Chamber_+ = \Chamber_-$. Every apartment $\PosApartment \times \NegApartment$ of $\Star\PosCell \times \Star\NegCell$ that contains $\PosChamber \times \NegChamber$ is isometric in an $\Height$-preserving way to an apartment $(\Another\PosApartment \intersect \Star\PosCell) \times (\Another\NegApartment \intersect \Star\NegCell)$ where $(\Another\PosApartment,\Another\NegApartment)$ is a twin apartment of $(\PosEuclBuilding,\NegEuclBuilding)$.
\end{cor}

\begin{proof}
If $L$ is the group from the Proposition, then $\Stabilizer{L}{\Chamber_+ \times \Chamber_-}$ acts transitively on apartments of $\Star\PosCell\times\Star\NegCell$ that contain $\Chamber_+ \times \Chamber_-$.
\end{proof}

\begin{lem}
\label{lem:descending_link_of_essential_cell_horizontal_coface_part}
Assume that $\Directions$ is rich. If $\Cell$ is essential, then $\Descending{\Horizontal\Link} \Cell$ is $(\dim \Horizontal\Link\Cell)$-spherical.
\end{lem}

\begin{proof}
If $\BigCell$ is a coface of $\Cell$, then there is a move $\Cell \Up \BigCell$ because $\BigCell \horizontal \Cell = \Cell\Min$ so that $\BigCell\Min = \Cell$ by Observation~\ref{obs:min_min}. Thus $\Depth \BigCell < \Depth \Cell$ and $\BigCell$ is descending unless $\max \Height|_\BigCell > \max \Height|_\Cell$. If this never occurs we are done, so we assume that there is a coface $\BigCell$ that is ascending.

We first analyze the situation inside an apartment: Write $\Cell = \PosCell \times \NegCell$ and let $\TwinApartment$ be a twin apartment that contains $\PosCell$ and $\NegCell$. As usual we make the identifications \eqref{dia:twin_identification}. The closed stars $\iota_+(\Star[\PosApartment] \PosCell)$ and $\iota_-(\Star[\NegApartment]\NegCell)$ are polytopes for which $\Directions$ is sufficiently rich by assumption. Let $\Polytope = (\Star[\PosApartment] \PosCell) \times (\Star[\NegApartment] \NegCell)$ and define $\pi \colon \PosApartment \times \NegApartment \to \EuclSpace$ by $\pi(\PosPoint,\NegPoint) = \iota_+(\PosPoint) - \iota_-(\NegPoint)$. By Observation~\ref{obs:sum_sufficiently_rich} $\Directions \union (\Directions + \Directions)$ is sufficiently rich for $\pi(\Polytope)$. Thus by Proposition~\ref{prop:sufficiently_rich_min_vertex} the set of points of $\pi(\Polytope)$ that have maximal distance to $\Zonotope[\Directions \union (\Directions + \Directions)]$ is a face $\Face$. By definition the set of $\Height$-maximal points of $\Polytope$ is just the inverse image of $\Face$ under $\pi$, which is again a face. And since we assumed $\Height$ to be non-constant on $\Polytope$, it is a proper face. Thus it defines a convex subset of $\Link\Cell$.

Now let $\PosChamber \ge \PosCell$ and $\NegChamber \ge \NegCell$ be chambers such that $\PosChamber = \Proj[\Star\PosCell]\NegChamber$ and $\NegChamber = \Proj[\Star\NegCell]\PosChamber$. By Corollary \ref{cor:apartments_that_contain_chamber_look_good}, every apartment of $\Link\Cell$ that contains $\Chamber \defeq (\PosChamber \direction \PosCell) * (\NegChamber \direction \NegCell)$ can identified with an apartment $(\PosApartment \intersect \Link\PosCell) * (\NegApartment \intersect \Link\NegCell)$ that comes from a twin apartment $\TwinApartment$ in such a way, that ascending cells are identified with ascending cells. Hence Proposition~\ref{prop:apartmentwise_coconvex_complexes} together with Proposition~\ref{prop:retract_set_onto_subcomplex} implies that $\Descending{\Horizontal\Link}\Cell$ is spherical.
\end{proof}

\begin{prop}
\label{prop:descending_links_are_spherical}
Assume that $\Directions$ is rich. Let $\Cell \subseteq \Space$ be a cell of positive height. If $\Cell$ is not essential, then the descending link of $\Subdivision\Cell$ is contractible. If $\Cell$ is essential then the descending link of $\Subdivision\Cell$ is spherical. Beyond every bound on the height there is an essential cell $\Cell$ such that the descending link of $\Subdivision\Cell$ is properly spherical.
\end{prop}

\begin{proof}
The statement about non-essential cells follows from Lemma~\ref{lem:descending_link_of_non-horizontal_cell} and Lemma~\ref{lem:descending_link_of_non-essential_cell}. The descending link of $\Subdivision\Cell$ where $\Cell$ is essential is a join of the face part, the vertical coface part, and the horizontal coface part. The face part is a sphere by Lemma~\ref{lem:descending_link_of_essential_cell_face_part}. The vertical coface part is a hemisphere complex by Lemma~\ref{lem:descending_link_of_essential_cell_vertical_coface_part} and thus spherical by Theorem~\ref{thm:hemisphere_complexes}. The horizontal coface part is spherical by Lemma~\ref{lem:descending_link_of_essential_cell_horizontal_coface_part}. At arbitrary height there are essential cells $\Cell$ that have empty horizontal coface part so that the descending link of $\Subdivision\Cell$ is not contractible.
\end{proof}

% ***************
% *** SECTION ***
% ***************

\section{Proof of the Main Theorem}
\label{sec:proof_of_main_theorem}

\begin{thm}
\label{thm:main_theorem_geometric_formulation}
Let $\EuclTwinBuilding$ be an irreducible locally finite Euclidean twin building of dimension $n$. Assume that $\Group$ acts strongly transitively on $\EuclTwinBuilding$ and that the kernel of the action is finite. Then $\Group$ is of type $\TopFin{2n-1}$ but not of type $\TopFin{2n}$.
\end{thm}

\begin{proof}
The set of heights of vertices $\Height(\Vertices\Space)$ is discrete. Thus we can write $\Subdivision\Height(\Vertices\Subdivision\Space) = \{r_0,r_1,\ldots\}$ which induces a filtration by spaces $\Space_i$, the full subcomplex of $\Space$ of vertices of height at most $r_i$. The space $\Space_{i+1}$ is obtained from $\Space_{i}$ by gluing in vertices $\Subdivision\Cell$ along their descending links, which by Proposition~\ref{prop:descending_links_are_spherical} are $(2n-1)$-spherical once $r_{i+1}$ is strictly positive in the first component. Moreover it is infinitely often not contractible.
%Hence the induced maps $\pi_k(\Space_i) \to \pi_k(\Space_{i+1})$ are isomorphisms for $k \le 2n-1$.

The action of $\Group$ on $\EuclTwinBuilding$ is codistance-preserving hence $\Group$ acts on each of the $\Space_i$.

The action of $\Group$ on pairs of chambers $(\PosChamber,\NegChamber)$ with $\PosChamber \op \NegChamber$ is transitive (by strong transitivity) so, since the twin building is locally finite, the action of $\Group$ on each of the $\Space_i$ is cocompact.

The stabilizer of two opposite chambers $\PosChamber$ and $\NegChamber$ in the automorphism group of $\EuclTwinBuilding$ is finite. This follows from \cite[Theorem~5.205]{abrbro} together with the fact that the twin building is locally finite. Since the twin building is locally finite, this implies that the stabilizer of any cell of $\Space$ is finite. Since the kernel of the action of $\Group$ is finite, the same is true of the stabilizers in $\Group$, so they are in particular of type $\TopFin{\infty}$.

The space $\Space$ is contractible because it is CAT($0$) as a product of the CAT($0$)-spaces $\PosEuclBuilding$ and $\NegEuclBuilding$.

Thus by Brown's criterion \cite[Corollary~3.3]{brown87} $\Group$ is of type $\TopFin{2n-1}$ but not of type $\TopFin{2n}$.
\end{proof}

The algebraic version is deduced using standard constructions:

\begin{thm}
\label{thm:main_theorem_algebraic_formulation}
Let $\Group$ be a group that admits a twin Tits system $(\Group,B_+,B_-,N,S)$ and set as usual $W = N/(N \intersect B_+)$. If $(W,S)$ is of irreducible affine type and rank $n+1$, $[B_\varepsilon s B_\varepsilon:B_\varepsilon]$ is finite for all $s \in S$ and $\varepsilon \in \{+,-\}$, and $\Intersect_{g \in G}gB_+g^{-1} \intersect \Intersect_{g \in G}gB_-g^{-1}$ is finite, then $\Group$ is of type $\TopFin{2n-1}$ but not of type $\TopFin{2n}$.

This is in particular the case if there is an RGD system $(\Group,(U_\alpha)_{\alpha\in\Phi},T)$ of type $(W,S)$, each $U_\alpha$ is finite, and $\Group_+ \defeq \gen{U_\alpha \mid \alpha \in \Phi}$ has finite index in $\Group$.
\end{thm}

\begin{proof}
By \cite[Theorem~6.87]{abrbro}, $\Group$ acts strongly transitively on a thick affine twin building $\EuclTwinBuilding$. Moreover the condition that $[B_\varepsilon s B_\varepsilon:B_\varepsilon]$ is finite for all $s \in S$ and $\varepsilon \in \{+,-\}$ implies that $\EuclTwinBuilding$ is locally finite (c.f.\ \cite[Section~6.1.7]{abrbro}). Finally $\Intersect_{g \in G}gB_+g^{-1} \intersect \Intersect_{g \in G}gB_-g^{-1}$ is the kernel of the action of $\Group$ on $\EuclTwinBuilding$. So we can apply Theorem~\ref{thm:main_theorem_geometric_formulation}.

An RGD system for $\Group$ gives rise to a twin BN pair by Theorem~8.80 of \cite{abrbro}. Moreover Theorem~8.81 of loc.cit.\ implies that the associated twin building is locally finite if the $U_\alpha$ are finite. And by Proposition~8.82 the centralizer of $\Group_+$ in $\Group$ is the kernel of the action of $\Group$ on the twin building; it is finite if $\Group_+$ has finite index.
\end{proof}

Finally we prove the Main~Theorem. Recall the setting: $\GroupScheme$ is a connected, non-commutative, absolutely almost simple $\FiniteField[q]$-group of rank $n>0$. The Main~Theorem states that $\Group = \GroupScheme(\FiniteField[q])$ is of type $\TopFin{2n-1}$ but not of type $\TopFin{2n}$.

\begin{proof}[Proof of the Main~Theorem]
Passing to an extension by a finite group, we may assume that $\GroupScheme$ is simply connected. By \cite[Proposition~10.4]{buxgrawit10} $\Group$ acts on a Euclidean twin building $\EuclTwinBuilding$. Moreover the halves of $\EuclTwinBuilding$ are $\Group$-equivariantly isometric to the buildings associated to $\GroupScheme(\FiniteField[q]((t))$ and $\GroupScheme(\FiniteField[q]((t^{-1})))$. That is the twin building is locally finite, of dimension $n$, and the kernel of the action is finite. Now the result follows from Theorem~\ref{thm:main_theorem_geometric_formulation}.
\end{proof}

\begin{rem}
Technically the proof of the Main~Theorem of course ``passes through'' Theorem~\ref{thm:main_theorem_algebraic_formulation}. But in lack of a better reference, it is easier to deduce it from Theorem~\ref{thm:main_theorem_geometric_formulation}.
\end{rem}

% ***************
% *** SECTION ***
% ***************

\section{Concluding remarks}

\begin{remnum}
The group $\Group$ is both, an $\Places$-arithmetic group over a global function field, which gives rise to a Bruhat--Tits building, and a Kac--Moody group, which gives rise to a twin building. We have used both of these facts in the proof of the Main~Theorem. The fact that the building is Euclidean is crucial: for example we use that hyperplanes and horospheres are the same thing inside an apartment. The twin building structure is also heavily used, but it seems that this is in fact just a convenient way to get around having to use the reduction theory developed by Harder in \cite{harder67, harder68, harder69} (cf.\ \cite{buxwor08} where reduction theory is used). This is to say that our proof of the Main~Theorem is better understood in the context of $\Places$-arithmetic groups then in the context of Kac--Moody groups. More concretely it should be possible to replace the metric codistance by a Morse function obtained from Harder's reduction theory and obtain a proof for $S$-arithmetic groups with $S$ arbitrary but there is little hope of generalizing it to hyperbolic Kac--Moody groups (see also Section~1 of \cite{buxgrawit10}).

On the other hand the approach that Abramenko has developed in \cite{abramenko96} (and not yet carried out for the entire Kac--Moody group) uses only the theory of twin buildings and generalizes directly to the compact hyperbolic case, but it does not produce results for $S$-arithmetic groups at other places.
\end{remnum}

The following result is certainly known to the experts but for lack of suitable reference we reproduce it here. It is the natural generalization of \cite[Proposition 2]{behr98}. The proof works as in the special case considered in \cite{abramenko96}.

\begin{prop}
\label{prop:increasing_places_preserves_topfin}
Let $\GlobalField$ be a global function field, $\GroupScheme$ an isotropic, connected, absolutely almost simple $\GlobalField$-group, and $\Places$ a non-empty, finite set of places of $\GlobalField$. If $\GroupScheme(\Integers[\Places])$ is of type $\TopFin{\SumOfLocalRanks}$ and $\Places' \supseteq \Places$ is a larger finite set of places, then $\GroupScheme(\Integers[\Places'])$ is also of type $\TopFin{\SumOfLocalRanks}$.
\end{prop}

\begin{proof}
Proceeding by induction it suffices to prove the case where only one place is added to $\Places$, i.e.\ $\Places' = \Places \union \{\Class{\Valuation}\}$ for some place $\Class{\Valuation}$. Also note that as far as finiteness properties are concerned, we may (and do) assume that $\GroupScheme$ is simply connected.

Let $\EuclBuilding_\Valuation$ be the Bruhat--Tits building that belongs to $\GroupScheme(\GlobalField_\Valuation)$ (see \cite{brutit72, brutit84}). The group $\GroupScheme(\Integers[\Places']) \subseteq \GroupScheme(\GlobalField_\Valuation)$ acts continuously on $\EuclBuilding_\Valuation$. We claim that this action is cocompact and that cell stabilizers are abstractly commensurable to $\GroupScheme(\Integers[\Places])$. With these two statements the result follows from Proposition~1.1 and Proposition~3.1 of \cite{brown87}.

Note that the stabilizer of a cell is commensurable to the stabilizers of its faces and cofaces since the building is locally finite (because the residue field of $\GlobalField$ is finite). Also all cells of same type are conjugate by the action of $\GroupScheme(\GlobalField_\Valuation)$. Hence it remains to see that some cell-stabilizer is commensurable to $\GroupScheme(\Integers[\Places])$. To see this note that $\GroupScheme(\Integers[\Valuation])$ is a maximal compact subgroup of $\GroupScheme(\GlobalField_\Valuation)$ (with respect to the topology given by the local field $\GlobalField_\Valuation$). The Bruhat--Tits Fixed Point Theorem \cite[Lemme 3.2.3]{brutit72} (see also \cite[Corollary~II.2.8]{brihae}) implies that it has a fixed point and by maximality it is the full stabilizer of the carrier $\Cell$ of that fixed point. (Actually the fixed point is a special vertex by the definition of $\EuclBuilding_\Valuation$ but we do not need this here.) Now $\GroupScheme(\Integers[\Places]) = \GroupScheme(\Integers[\Places']) \intersect \GroupScheme(\Integers[\Valuation])$ so $\GroupScheme(\Integers[\Places])$ is the stabilizer in $\GroupScheme(\Integers[\Places'])$ of $\Cell$.

For cocompactness we use that $\GroupScheme(\Integers[\Places'])$ is dense in $\GroupScheme(\GlobalField_\Valuation)$, see Lemma~\ref{lem:density_of_s-arith_subgroup} below. Let $\Point$ be an interior point of some chamber of $\EuclBuilding_\Valuation$. The orbit $\GroupScheme(\GlobalField_\Valuation).\Point$ is a discrete space which, by strong transitivity, contains one point from every chamber of $\EuclBuilding_\Valuation$. The orbit map $\GroupScheme(\GlobalField_\Valuation) \to \GroupScheme(\GlobalField_\Valuation).x$ is continuous by continuity of the action, so the image of the dense subgroup $\GroupScheme(\Integers[\Places'])$ is dense in the discrete space $\GroupScheme(\GlobalField_\Valuation).x$. Hence $\GroupScheme(\Integers[\Places'])$ acts transitively on chambers and in particular cocompactly.
\end{proof}

The proposition allows to deduce the following consequence of \cite[Main~Theorem]{buxgrawit10} and our Main~Theorem:

\begin{cor}
Let $\GroupScheme$ be a connected, absolutely almost simple $\FiniteField[q]$-group of rank $\Dimension > 0$ and let $\Places$ be a finite set of places of $\FiniteField[q](t)$. If $\Places$ contains the place of the $t$-adic valuation $\Valuation_0$ or of the valuation at infinity $\Valuation_\infty$, then $\GroupScheme(\Integers[\Places])$ is of type $\TopFin{n-1}$. If $\Places$ contains both, $\Class{\Valuation_0}$ and $\Class{\Valuation_\infty}$, then $\GroupScheme(\Integers[\Places])$ is of type $\TopFin{2n-1}$.
\end{cor}

It remains to provide the density statement used in the proof. It is known and a consequence of the Strong Approximation Theorem:

\begin{lem}
\label{lem:density_of_s-arith_subgroup}
Let $\GlobalField$ be a global field and let $\GroupScheme$ be an isotropic, connected, simply connected, absolutely almost simple $\GlobalField$-group. Let $\Places$ be a non-empty finite set of places and let $\Class{\Valuation} \nin \Places$. Then $\GroupScheme(\Integers[\Places \union \{\Class{\Valuation}\}])$ is dense in $\GroupScheme(\GlobalField_\Valuation)$.
\end{lem}

\begin{proof}
For a valuation $\Valuation$ of $\GlobalField$ let $\GlobalField_\Valuation$ denote the local field at $\Valuation$ and $\Integers[\Valuation]$ the ring of integers in $\GlobalField_\Valuation$. For a finite set $\Places$ of places of $\GlobalField$ let $\Adeles_\Places = \prod_{\Valuation \in \Places} \GlobalField_\Valuation \times \prod_{\Valuation \nin \Places} \Integers[\Valuation]$ denote the ring of $\Places$-adeles. Recall that the ring of adeles is $\Adeles = \lim_{\Places} \Adeles_\Places$ (see \cite{weil82}).

Note that $\GroupScheme_\Places \defeq \prod_{s \in S} \GroupScheme(\GlobalField_\Valuation)$ is non-compact by \cite[Proposition~2.3.6]{margulis}.

Recall that $\GlobalField_\Valuation$ embeds into $\Adeles$ at $\Valuation$, and that $\GlobalField$ discretely embeds into $\Adeles$ diagonally. With these identifications $\GroupScheme(\GlobalField) \cdot \GroupScheme_\Places$ is dense in $\GroupScheme(\Adeles)$ by \cite[Theorem~A]{prasad77}, that is if $U$ is an open subset of $\GroupScheme(\Adeles)$, then $\GroupScheme(\GlobalField) \intersect U \GroupScheme_\Places \ne \emptyset$.

If $V$ is an open subset of $\GroupScheme(\GlobalField_\Valuation)$, then
\[
U = V \times \prod_{\Valuation' \in \Places} \GroupScheme(\GlobalField_{\Valuation'}) \times \prod_{\Valuation' \nin \Places \union \{\Valuation\}} \GroupScheme(\Integers[\Valuation'])
\]
is open in $\GroupScheme(\Adeles)$. Hence there is a $g \in \GroupScheme(\GlobalField)$ with $g \in V$ and $g \in \GroupScheme(\Integers[\Places \union \{\Valuation\}])$ (where we now consider $\GroupScheme(\GlobalField)$ and $\GroupScheme(\Integers[\Places \union \{\Valuation\}])$ as subgroups of $\GroupScheme(\GlobalField_\Valuation)$). Thus $V \intersect \GroupScheme(\Integers[\Places \union \{\Valuation\}]) \ne \emptyset$ as desired.
\end{proof}

It may be interesting to note that the direction of the implication of Proposition~\ref{prop:increasing_places_preserves_topfin} is inverse to the one that is implied by Tiemeyer's criterion \cite[Theorem~3.1]{tiemeyer97} in the number field case:

\begin{thm}[Tiemeyer]
\label{thm:tiemeyer_criterion}
Let $\GlobalField$ be a global number field, $\GroupScheme$ a $\GlobalField$-group, and $\Places$ a set of places of $\GlobalField$. If $\GroupScheme(\Integers[\Places])$ is of type $\TopFin{\SumOfLocalRanks}$ and $\Places' \subseteq \Places$, then $\GroupScheme(\Integers[\Places'])$ is also of type $\TopFin{\SumOfLocalRanks}$.
\end{thm}

\bibliographystyle{amsalpha}
\bibliography{../local}

\end{document}